\pdfoutput=1
\documentclass[11pt]{amsart}
\usepackage{geometry}                
\usepackage[parfill]{parskip}    
\usepackage{graphicx,color,amsthm}
\usepackage{amssymb,enumitem,multicol}
\usepackage{epstopdf}
\usepackage{pdfsync}
\usepackage{fullpage}
\theoremstyle{plain} 

\newtheorem{thm}{Theorem}

\newtheorem{lem}[thm]{Lemma} 
\newtheorem{prop}[thm]{Proposition}

\theoremstyle{definition}

\newtheorem{innercustomgeneric}{\customgenericname}
\providecommand{\customgenericname}{}
\newcommand{\newcustomtheorem}[2]{%
  \newenvironment{#1}[1]
  {%
   \renewcommand\customgenericname{#2}%
   \renewcommand\theinnercustomgeneric{##1}%
   \innercustomgeneric
  }
  {\endinnercustomgeneric}
}

\newcustomtheorem{customthm}{Theorem}
\newcustomtheorem{customlemma}{Lemma}
\newcustomtheorem{customprop}{Proposition}
\title{Linking numbers in three-manifolds}
\author{Patricia Cahn and Alexandra Kjuchukova}
\begin{document}
\maketitle

\begin{abstract}  Let $M$ be a connected, closed, oriented three-manifold and $K$, $L$ two rationally null-homologous oriented simple closed curves in $M$. We give an explicit algorithm for computing the linking number between $K$ and $L$ in terms of a presentation of $M$ as an irregular dihedral $3$-fold cover of $S^3$ branched along a knot $\alpha\subset S^3$. Since every closed, oriented three-manifold admits such a presentation, our results apply to all (well-defined) linking numbers in all three-manifolds. Furthermore, ribbon obstructions for a knot $\alpha$ can be derived from dihedral covers of $\alpha$. The linking numbers we compute are necessary for evaluating one such obstruction. This work is a step toward testing potential counter-examples to the Slice-Ribbon Conjecture, among other applications.

\end{abstract}

\section{Introduction}

The notion of linking number between knots in $S^3$ dates back at least as far as Gauss~\cite{gauss1867allgemeine}. More generally, given a closed, oriented three-manifold $M$ and two rationally null-homologous, oriented, simple closed curves  $K, L\subset M$, the linking number $lk(K, L)$ is defined as well. It is given by $\dfrac{1}{n}\left(K\cdot 
C_L \right)$, where $C_L$ is a 2-chain in $M$ with boundary $nL$, $n\in \mathbb{N}$, and $\cdot$ denotes the signed intersection number. This linking number is well-defined and symmetric~\cite{birman1980seifert}.

Let the three-manifold $M$ be presented as a 3-fold irregular dihedral branched cover of $S^3$, branched along a knot.  Every closed oriented three-manifold admits such a presentation~\cite{hilden1974every, hirsch1974offene, montesinos1978}. 
Consider a branched cover $f:M\to S^3$ of this type, and let $\gamma, \delta \subset S^3$ be oriented, closed curves embedded disjointly from each other and from the branching set $\alpha$ of $f$.  In Theorem~\ref{main-thm}, we give a formula for the linking number in $M$ between any two connected components of the pre-images of $\gamma$ and $\delta$, in the case where the pre-images of $\gamma$ and $\delta$ have three connected compents each.  The general case is given in Section~\ref{fewer-lifts}.  This linking number is computed in terms of a diagram of the link $\alpha\cup\gamma\cup\delta$. The geometric construction underlying the computation is reviewed in Section~\ref{algoverview.subsec} and serves to complement the theorem statement, which is combinatorial in flavor. Linking numbers in dihedral branched covers of $S^3$ are needed for calculating several knot and three-manifold invariants~\cite{litherland1980formula, CS1975invariants, kjuchukova2018dihedral, geske2018signatures}; applications are considered in Section~\ref{applications.subsec}.

Briefly, our technique is the following. The cone on the link $\alpha \cup \gamma \cup \delta$  gives a cell structure on $S^3$ which lifts, via the map $f$, to a cell structure on $M$.   Two-chains bounding closed connected components of $f^{-1}(\gamma)$ and $f^{-1}(\delta)$ are found by solving a system of linear equations. We obtain these equations by examining the diagram of $\alpha \cup \gamma \cup \delta$ used to construct the cell structure on $M$. Finally, intersection numbers between lifts of $\delta$ and the 2-chains bounding lifts of $\gamma$ are computed from local data about the relevant 1- and 2-cells.

Classically, a knot invariant is derived from linking numbers in branched covers as follows.  
Let $\alpha\subset S^3$ be a Fox $3$-colorable knot. Any 3-coloring of $\alpha$ determines an irregular dihedral $3$-fold covering map $f: M\to S^3$ with branching set $\alpha$, as reviewed in Section~\ref{dihedral.subsec}. Given such a three-fold cover $f$, the preimage of the branching set, $f^{-1}(\alpha)$, has two connected components whose linking number, in $M$, is either a rational number or undefined. The set of these linking numbers over all distinct 3-colorings of $\alpha$ is called the {\it linking number invariant} of $\alpha$. Analogous invariants can be derived for more general Fox $p$-colorings and other types of branched covers.

Dihedral linking numbers have been instrumental in distinguishing and tabulating knots, including in various situations where other invariants do not suffice. The linking number invariant was introduced by Reidemeister in~\cite{reidemeister1929knoten}, where he applied it to tell apart two knots with the same Alexander polynomial. In~\cite{riley1971homomorphisms}, Riley generalized this idea and used linking numbers in 5-fold (non-dihedral) branched covers to distinguish a pair of mutants whose Alexander polynomials were trivial. Two 36-crossing knots with the same Jones polynomial were distinguished by Birman using linking numbers in four-fold simple branched covers~\cite{birman1985jones}. 

Linking numbers in dihedral branched covers are also good for studying certain properties of knots:  they provide an obstruction to amphichirality~\cite{fox1970metacyclic, perko1974classification} and invertibility~\cite{hartley1983identifying}. But the most well-known story is perhaps that of the Perko Pair, which consists of ``two" knots which dihedral linking numbers failed to distinguish. These knots turned out to be isotopic, and constituted an accidental duplicate in Conway's knot table~\cite{conway1970enumeration}. The mistake was corrected by Perko.  His discovery also provided a counterexample to a conjecture of Tait -- stating that two reduced alternating diagrams of a given knot have equal the writhe -- previously believed to be established as a theorem.

Historically, efforts at knot classification have relied heavily on linking numbers in branched covers. Bankwitz and Schumann~\cite{bankwitz1934viergeflechte}  classified knots of up to nine crossings using linking numbers in dihedral covers of 2-bridge knots as their primary tool. (Note that the irregular dihedral branched cover of a 2-bridge knot is always $S^3$; a proof of this old observation is recalled in \cite{kjuchukova2016classification}.) Perko extended these methods, which allowed him to complete the classification to knots of ten and eleven crossings~\cite{perko1974classification}. {Burde proved that dihedral linking numbers can tell apart all 2-bridge knots~\cite{burde1988links}, without regard to crossing number.} The largest-scale computation of linking numbers was done by Dowker and Thistlethwaite, who succeeded in tabulating millions of knots~\cite{dowker1982classification}.   
It is difficult to imagine that today's knot tables would be as advanced in the absence of Reidemeister's extremely powerful idea to consider linking numbers between the branch curves in non-cyclic branched covers of knots. For a more detailed account of the role of linking numbers in knot theory, as well as several illuminating examples, see~\cite{perko2016historical}.
  
Our results extend the classical linking number computation to include linking numbers of curves other than the branch curves, namely, closed connected components of $f^{-1}(\gamma)$ and $f^{-1}(\delta)$, where $\gamma, \delta\subset S^3$ are curves in the complement of the branching set.  It is helpful to formally regard points on $\gamma$ and $\delta$ as points on the branching set of $f$, with the property that each of their pre-images has branching index 1. Accordingly, we refer to $\gamma$ and $\delta$ as a {\it pseudo-branch curves} of $f$. We will call each closed connected component of $f^{-1}(\gamma)$ (resp. $f^{-1}(\delta)$) a {\it lift} of $\gamma$ (resp. $\delta$).  Finally, despite the apparent ambiguity, we will also use the phrase ``pseudo-branch curves" to refer to the lifts themselves. Since every closed, connected, oriented 3-manifold admits a presentation as a 3-fold dihedral cover of $S^3$ branched along a knot, our methods compute all well-defined linking numbers in all 3-manifolds; this is proved at the end of Section~\ref{notation.sec}.

\subsection{Algorithm overview and the main theorem}\label{notation.sec}
\label{algoverview.subsec} We now summarize the geometric setup underlying our computation, and state our main theorem. Let $\alpha\subset S^3$ be a 3-colored knot and $f: M\to S^3$ be the corresponding dihedral cover of $S^3$ branched along $\alpha$.  Let $\gamma$, $\delta\subset S^3 - \alpha$ be two disjoint, oriented knots. We treat the homomorphism $\rho: \pi_1(S^3-\alpha)\twoheadrightarrow D_3$ from which the branched cover $f$ arises as a homomorphism of $\pi_1(S^3-\alpha-\gamma-\delta)$ in which meridians of $\gamma$ and $\delta$ all map to the trivial element; accordingly, we refer to $\gamma$, $\delta$ as {\it pseudo-branch curves}.  We compute linking numbers between connected components of $f^{-1}(\gamma)$ and  $f^{-1}(\delta)$ by the following procedure.  

\begin{enumerate}
	\item Choose a cell structure on $S^3$ determined by the cone on the link $\alpha \cup \gamma\cup \delta$; see Figure~\ref{coneonlink.fig}.
	\item Lift this cell structure to $M$ by examining the lifts of the cells near each crossing of the link diagram downstairs; see, for example, Figure \ref{upstairsinhomog.fig}.  This cell structure contains the lifts of the pseudo-branch curves as 1-subcomplexes.
	\item Solve a linear system to determine which of the lifts of the pseudo-branch curves are rationally null-homologous.  For each rationally null-homologous lift of a pseudo-branch curve, we find an explicit 2-chain which it bounds. 
	\item For each pair of rationally null-homologous lifts of the pseudo-branch curves, we compute linking numbers by adding up the signed intersection numbers of the relevant 1- and 2-cells. 
\end{enumerate}

	  \begin{figure}[htbp]\includegraphics[width=3in]{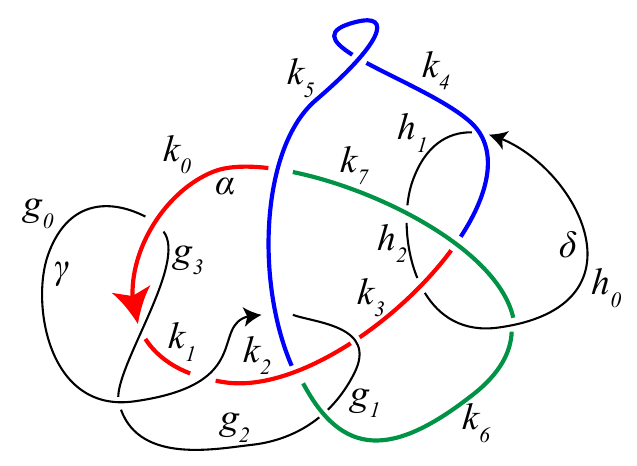}
	  	\caption{The labelled arcs of the link diagram $\alpha\cup\gamma\cup\delta$.}
	  \end{figure}
	  
	  \begin{figure}[htbp]\includegraphics[width=4.5in]{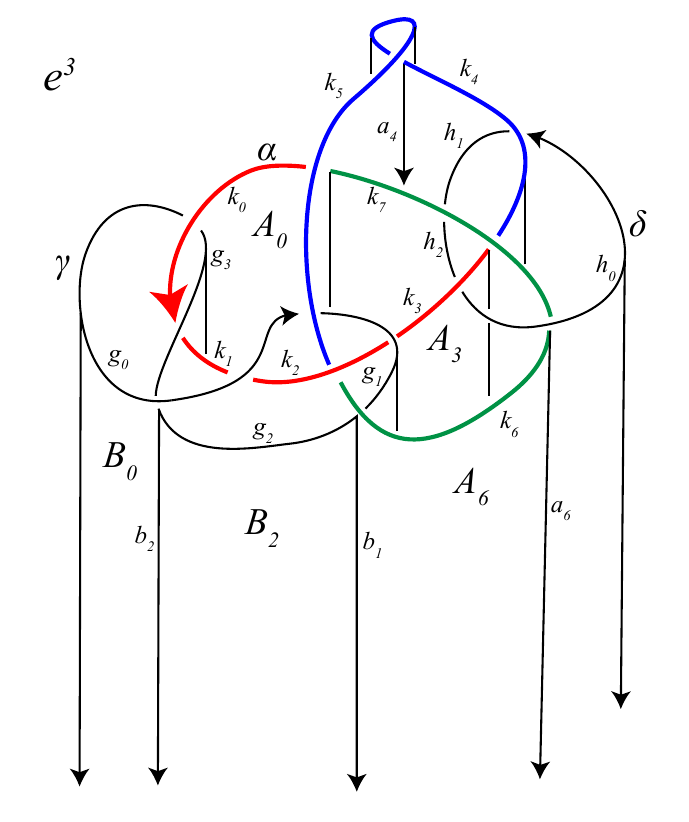}
	  	\caption{A cell structure on $S^3$ determined by the cone on the link $\alpha\cup\gamma\cup\delta$ together with the notation for the individual cells.  }
	  	\label{coneonlink.fig}
	  	\end{figure}

 Steps (1) and (2) are discussed in Section~\ref{alg}.  Step (3) is carried out in Proposition~\ref{system-eqns2}, which determines when a lift of a pseudo-branch curve bounds a 2-chain, and finds the 2-chain when it exists.  Step (4) is the content of Theorem~\ref{main-thm}, which gives a formula for the linking number between lifts of pseudo-branch curves.
	
 We now state our main results, Proposition~\ref{system-eqns2} and Theorem~\ref{main-thm}. We assume for the moment that each of the pseudo-branch curves has three (closed, connected) lifts, and denote these by  $\gamma_j$ and $\delta_k$, $j, k\in \{1, 2, 3\}$. Both $\gamma_j$ and $\delta_k$ must be rationally null-homologous for their linking number to be well-defined.  We verify this condition by reversing the roles of $\gamma$ and $\delta$ in our computations and thus making sure that each of the curves bounds a 2-chain. The lift $\gamma_j$ is rationally null-homologous if and only if a solution $(x_0^j, x_1^j,\dots,x_{m-1}^j)\in\mathbb{Q}^m$ to the system of equations in Proposition~\ref{system-eqns2} exists.  The $x^j_i$ describe a rational 2-chain with boundary $\gamma_j$,  namely they are coefficients for the 2-cells $A_{2,i}$ and $-A_{3,i}$ in the chain (these 2-cells are defined in Section~\ref{liftcells.sec}).  Additional notation is given in Table~\ref{notation.tab}.  The precise definitions of items 10 to 14 in the table are technical and given in the equations listed, which can be found in Sections~\ref{M-cells} and~\ref{linking-k}.

 \begin{prop} 
 \label{system-eqns2}
 \label{thm2}
 Let $s$ denote the number of crossings of $\gamma$ under $\alpha$ plus the number of self-crossings of $\gamma$, let $m$ denote the number of crossings of $\alpha$ under $\gamma$ plus the number $n$ of self-crossings of $\alpha$.  Let $f(i)$ denote the index of the overstrand $k_{f(i)}$ at crossing $i$, and let the signs $\epsilon$, and $\epsilon_k$ for $k=1,2,3,4$ be as in Table~\ref{notation.tab}.  If the following inhomogeneous system of linear equations 

 $$\left\{\begin{array}{ll}
 x_i^j-x_{i+1}^j+\epsilon_1(i)\epsilon_2(i)x_{f(i)}^j=0 &\text{ if crossing i of } \alpha \text{ is inhomogeneous}\\
 x_i^j-x_{i+1}^j+2\epsilon_3(i)x_{f(i)}^j=0&\text{ if crossing i of }\alpha \text{ is homogeneous}\\
 x^j_i-x^j_{i+1}=\epsilon(i)\epsilon^j_4(i)&\text{if strand i of }\alpha \text{ passes under }\gamma
 \end{array}
 \right.$$

 has a solution $(x_0^j,x_1^j,\dots,x_{m-1}^j)$ over $\mathbb{Q}$, then the lift $\gamma_j$ of $\gamma$ is rationally null-homologous and is bounded by the 2-chain
 $$C_j=\sum_{i=0}^{s-1} B_{j,i}+\sum_{i=0}^{m-1} x_i^j(A_{2,i}-A_{3,i}).$$  

 \end{prop}

  Let $I_{j,k}$ be the linking number of  $\gamma_j$ and $\delta_k$.  Theorem~\ref{main-thm} gives a formula for $I_{j,k}$ in terms of the solution to the system of equations in Proposition~\ref{system-eqns2}.
  
\begin{thm}
\label{main-thm}
Let $f: M\to S^3$ be a three-fold irregular dihedral cover branched along a knot $\alpha$, and let $\gamma, \delta\subset S^3-\alpha$.  If the lifts $\gamma_j$ and $\delta_k$ are rationally null-homologous closed loops in $M$ for $j,k\in\{1,2,3\}$, then the linking number $I_{j,k}$ of  $\gamma_j$ with $\delta_k$ is the sum:

$$I_{j,k}=\sum_{i=0}^{t-1} c_i,$$
where $c_i$ is given by

$$c_i =\left\{\begin{array}{ll}
\epsilon_5^k(i) x^j_{f(i)}&\text{if }h_i \text{ terminates at an arc } k_{f(i)} \text{ of }\alpha;\\
\epsilon_\delta(i)\epsilon_6^{j,k}(i)&\text{if }h_i \text{ terminates at an arc of }\gamma;\\
0&\text{if }h_i\text{ terminates at an arc of }\delta.\\
\end{array}
\right.$$
      \end{thm}

\begin{table}[htbp]
\begin{center}
	\begin{tabular}{|l|l|l|}
		\hline
		1.&$k_0,k_1,\dots,k_{m-1} $&Arcs of $\alpha$ in diagram $\alpha \cup \gamma$  \\
		\hline
		2.&$g_0,g_1,\dots,g_{s-1}$&Arcs of $\gamma$ in diagram $\alpha \cup \gamma $ \\
		\hline
		3.&$h_0, h_1, \dots, h_{t-1}$&Arcs of $\delta$ in diagram $\alpha \cup \gamma \cup \delta$  \\
		\hline
		4.&$\epsilon(i)\in\{-1,1\}$&Local writhe number at the head of arc $k_i$  \\
		\hline
		5.&$\epsilon_\gamma(i)\in\{-1,1\}$&Local writhe number at the head of arc $g_i$  \\
		\hline
		6.&$\epsilon_\delta(i)\in\{-1,1\}$&Local writhe number at the head of arc $h_i$  \\
		\hline
		7.&$f(i)$& Subscript of overcrossing arc at head of arc $k_i$  \\
		\hline
		8.&$f_\gamma (i)$& Subscript of overcrossing arc at head of arc $g_i$  \\
		\hline
		9.&$f_\delta (i)$& Subscript of overcrossing arc at head of arc $h_i$  \\
		\hline
		10.&$\epsilon_1(i)\in\{-1,1\}$& Concerns 2-cells above inhomogeneous crossing of $\alpha$. See Equation~\ref{epsilon1}\\ 
		\hline
		11.&$\epsilon_2(i)\in\{-1,1\}$& Concerns 2-cells above inhomogeneous crossing of $\alpha$. See Equation~\ref{epsilon2}\\ 
		\hline
		12.&$\epsilon_3(i)\in\{-1,1\}$& Concerns 2-cells above a homogeneous crossing of $\alpha$. See Equation~\ref{epsilon3}\\
		\hline
		13.&$\epsilon_4^j(i)\in\{-1,0,1\}$& Concerns 2-cells above a crossing of $\alpha$ under $\gamma$. See Equation~\ref{epsilon4}\\
		\hline
		13.&$\epsilon_5^k(i)\in\{-1,0,1\}$& Concerns 2-cells above a crossing of $\delta$ under $\alpha$. See Equation~\ref{epsilon5}\\
		\hline
		14.&$\epsilon_6^{j,k}(i)\in\{0,1\}$& Concerns 2-cells above a crossing of $\delta$ under $\gamma$. See Equation~\ref{epsilon6}\\
		\hline
	\end{tabular}
	\vskip .1in
	\caption{Notation}\label{notation.tab}
\end{center}
\end{table}

We have focused here on the case where each pseudo-branch curve lifts to three closed loops because this case is the one we encounter exclusively in our main application~\cite{cahnkjuchukova2018computing}. 
In general, the number of connected components of $f^{-1}(\gamma)$ is determined by the image of $[\gamma]\in\pi_1(S^3-\alpha)$ under the homomorphism $\pi_1(S^3-\alpha)\to D_3$ which determines the branched cover $f$. Therefore, the number of components of $f^{-1}(\gamma)$ can be determined from the the link diagram $\alpha\cup\gamma$ where $\alpha$ is 3-colored.
Computations involving pseudo-branch curves whose pre-images under the branched covering map consist of fewer than three connected components can be carried out using the same techniques; see Section~\ref{fewer-lifts}.  
Theorem~\ref{main-thm} can also be used to compute linking numbers between the branch curves themselves, as well as linking numbers between branch and pseudo-branch curves, as discussed in Sections \ref{brlinking.sec} and \ref{brpblinking.sec}.

Our methods compute all well-defined linking numbers in all closed, connected, oriented 3-manifolds.

\begin{lem}
	\label{alllinkingnumbers.lem} Let $M$ be a closed, connected, oriented 3-manifold, and let $K\cup L$ be a 2-component oriented link in $M$. Denote by $f: M\rightarrow S^3$ a 3-fold irregular dihedral cover whose branching set is the knot $\alpha\subset S^3$. 
	{Then $K\cup L$ is isotopic to a link $K'\cup L'$ such that $f(K'\cup L')$ is a link disjoint from $\alpha$.} 
\end{lem}
This lemma follows from a standard general position argument. See, for example,~\cite{mulazzani1998representing}, in which the authors give a diagrammatic theory for links in 3-manifolds represented as 3-fold covers of $S^3$.  In particular, their labelled Reidemeister moves provide an alternative approach to computing linking numbers between   lifts of pseudo-branch curves. 

Given rationally null-homologous $K$, $L$ as in the above lemma, note that $lk(K, L)=lk(K', L')$, since the two links $(K, L)$ and $(K', L')$ are isotopic.  Now let $\gamma=f(K')$ and $\delta=f(L')$.  That is, $K'$ and $L'$ are closed connected components of $f^{-1}(\gamma)$ and $f^{-1}(\delta)$, respectively. In the language of this paper, $f^{-1}(\gamma)$ and $f^{-1}(\delta)$ are lifts of the pseudo-branch curves $\gamma$ and $\delta$.  If $\gamma$ and $\delta$ each have three lifts,  the linking number of $K'$ and $L'$ can be computed by the formula given in Theorem~\ref{main-thm}, yielding the linking number of $K$ and $L$.  Otherwise the linking number can be computed as in Section~\ref{fewer-lifts}.

\subsection{Applications to branched covers of 4-manifolds and the Slice-Ribbon Conjecture.}\label{applications.subsec}   In~\cite{CS1975invariants}, Cappell and Shaneson gave a formula, in terms of linking numbers of lifts of pseudo-branch curves,  for the Rokhlin $\mu$ invariant of a dihedral cover of knot $\alpha$. As noted earlier, every oriented three-manifold is a dihedral cover of a knot~\cite{hilden1974every, hirsch1974offene, montesinos1978}; hence, this method is universal. 
Secondly, Litherland~\cite{litherland1980formula} showed that Casson-Gordon invariants of a knot can also be computed using linking numbers of pseudo-branch curves in a branched cover. 
The algorithm provided herein allows for the execution of a key missing step in evaluating Casson-Gordon and Rokhlin $\mu$ invariants via the above methods.

 The application we focus on is the computation of a ribbon obstruction $\Xi_p$ arising in the study of singular dihedral branched covers of four-manifolds.  In~\cite{kjuchukova2016classification}, the second author gives a formula for the signature of a $p$-fold irregular dihedral branched cover $f: Y\to X$ between closed oriented topological four-manifolds $X$ and $Y$, in the case where the branching set $B$ of $f$ is a closed oriented surface embedded in the base $X$ with a cone singularity described by a knot $\alpha\subset S^3$. This formula shows that the signature of $Y$ deviates from the locally flat case by a defect term, $\Xi_p(\alpha)$, which is determined by the singularity $\alpha$. The term  $\Xi_p(\alpha)$ can be calculated in part via linking numbers of pseudo-branch curves in a dihedral cover of $\alpha$. If the base $X$ of the covering map $f$ is in fact $S^4$, the signature of the cover $Y$ is exactly equal to $\Xi_p(\alpha)$. In particular, our method for computing linking numbers between pseudo-branch curves allows us to determine the signature of a dihedral branched cover  of $S^4$ in terms of combinatorial data about the singularity on the branching set. We give an example of such a computation, using the algorithm given in this paper, in~\cite{cahnkjuchukova2018computing}. Furthermore, for a slice knot $\alpha$, the integer $\Xi_p(\alpha)$ can be used to derive an obstruction to $\alpha$ being homotopy ribbon~\cite{cahnkjuchukova2017singbranchedcovers, geske2018signatures}. Precisely, for a fixed $p$, $\Xi_p(\alpha)$ is constrained in a fixed range for all homotopy ribbon knots. This obstruction provides a new method to test counter-examples to the Slice-Ribbon conjecture. We are interested in applying the results of this paper to search for a slice knot which is not ribbon; we use our algorithm to compute $\Xi_3$ for concrete examples of slice knots in \cite{cahnkjuchukova2018computing}.  {In~\cite{cahnkju2018genus} we give an infinite family of knots whose four-genus is computed with the help of the $\Xi_3$ invariant.} An effective method for evaluating linking numbers in three-manifolds is essential for using the $\Xi_p(\alpha)$ invariant to study knot four-genus and knot concordance.   

\subsection{Overview of the article.} In Section~\ref{alg}, we recall the definition of an irregular dihedral cover, and we discuss the relevant cell structure on $S^3$, as well its lift to the cover $M$. In Section~\ref{chains.sec} we construct the rational 2-chains bounding the pseudo-branch curves, proving Proposition~\ref{system-eqns2}.  In Section~\ref{epsilons}, we prove Theorem~\ref{main-thm}, which gives the formula for the linking numbers between lifts of pseudo-branch curves, as well as Theorem~\ref{pb-b}, which gives an analogous formula for the linking numbers between  lifts of a pseudo-branch curve and a branch curve.  Section~\ref{ex}  illustrates our algorithm on a concrete example of a three-fold dihedral cover and several pseudo-branch curves therein. Due to the large number of cells used, computations by hand quickly evolve into an unwieldy task, even for the most resolute and concentrated persons. Our algorithm for calculating linking numbers in branched covers has therefore been implemented in Python.  The code is included in the Appendix\footnote{Appendix not included in the published version.}.

{\bf Acknowledgements.}  Parts of this work were completed at the Max Planck Institute for Mathematics.  We thank MPIM for its support and hospitality.  We are grateful to Julius Shaneson for contributing ideas to this paper. Thanks also to Ken Perko for his feedback on the first version of our manuscript. This work was partially supported by the Simons Foundation/SFARI (Grant Number 523862, P. Cahn) and by NSF grants DMS 1821212 and DMS 1821257 to the authors.

\section{A combinatorial method for computing linking numbers}
\label{alg}

\subsection{Irregular dihedral covers}\label{dihedral.subsec} Let $\alpha$ be a knot in $S^3$ and $f: M\to S^3$ a covering map branched along $\alpha$. The branched cover $f$ is determined by its unbranched counterpart, $f_{|f^{-1}(S^3-\alpha)}$. Thus, we can associate to it a group homomorphism  $\rho:\pi_1(S^3-\alpha)\rightarrow G$ for some group $G$. For us, $G$ is always $D_{p}$, the dihedral group of order $2p$, $\rho$ is surjective, and $p$ is odd. The homomorphism $\rho$ induces the regular $2p$-fold dihedral cover of $(S^3, \alpha)$; this cover corresponds to the subgroup $\ker \rho\subset\pi_1(S^3-\alpha)$. The irregular $p$-fold dihedral cover of $(S^3, \alpha)$, also induced by $\rho$, corresponds to a subgroup $\rho^{-1}(\mathbb{Z}_2)\subset\pi_1(S^3-\alpha)$, where $\mathbb{Z}_2$ can be any subgroup of $D_{p}$ of order 2. The irregular dihedral cover is a $\mathbb{Z}_2$ quotient of the regular one, and different choices of subgroup $\mathbb{Z}_2\subset D_{p}$ correspond to different choices of an involution.  Recall also that $\rho$ can be represented by a $p$-coloring of the knot diagram, where the ``color" of each arc indicates the element in $D_p$ of order 2 to which $\rho$ maps the element of the knot group corresponding to the meridian of this arc.  In this paper we focus on 3-fold irregular dihedral covers.  The colors 1, 2, and 3 correspond to the transpositions $(23)$, $(13)$, and $(12)$ respectively.  Given a 3-fold dihedral cover, the pre-image of the knot $\alpha$ has two connected components $\alpha_1$ and $\alpha_2$, with branching indices 1 and 2 respectively.

\subsection{The cell structure on $S^3$} 
\label{M-cells}

This section serves primarily to describe the cell structure on $S^3$ determined by the cone on the link $\alpha\cup\gamma\cup \delta$, and some related notation. This is a subdivision of the cell structure used by Perko~\cite{perko1964thesis} to compute the linking number of the branch curves $\alpha_1$ and $\alpha_2$. 

We begin by focusing on the part of the cell structure determined only by the cone on $\alpha\cup \gamma$, as the lifts of these 2-cells are sufficient to construct a 2-chain bounding each lift of $\gamma$. The relevant notation is summarized in Table~\ref{notation.tab}.

The arcs of $\alpha$ in the link diagram of $\alpha \cup \gamma$ are labelled $k_0$, $k_1$, $\dots$, $k_{m-1}$, proceeding along the diagram in the direction of the orientation of $\alpha$;  $m$ is the sum of the number of crossings of $\alpha$ with itself and the number of crossings of $\alpha$ with $\gamma$ where $\alpha$ passes under $\gamma$.  For the purposes of labeling the lifts of 2-cells in a systematic way, we require that the diagram of $\alpha$ have an even number of crossings. We can arrange this to be the case by performing a Type 1 Reidemeister move on $\alpha$, if necessary. From now on, we assume without further comment that the diagram of $\alpha$ has this property. Similarly, the arcs of $\gamma$ are labelled $g_0$, $g_1$, $\dots$, $g_{s-1}$, where $s$ is the number of crossings of $\gamma$ with itself plus the number of crossings of $\alpha$ with $\gamma$ where $\gamma$ passes under $\alpha$. We refer to the crossing at the head of arc $k_i$ as the $i^{th}$ crossing of $\alpha$, and the crossing at the head of the arc $g_i$ as the $i^{th}$ crossing of $\gamma$; in each case the over-arc could be an arc of $\alpha$ or $\gamma$. After the arcs $k_i$ and $g_i$ have been labelled, we introduce the third link component $\delta$ to the diagram, and label its arcs $h_0$, $h_1$, $\dots$, $h_{t-1}$.  If several consecutive arcs of $\delta$ are separated by over-arcs of $\delta$, we treat these arcs as a single long arc with one label $h_i$, so $t$ above is the number of crossings of $\delta$ under $\alpha$ plus the number of crossings of $\delta$ under $\gamma$ (this allows us to slightly simplify the input to the computer program).  

{We denote by $\epsilon(i)$, $\epsilon_\gamma(i)$, or $\epsilon_\delta(i)$ the local writhe number at the head of $k_i$, $g_i$, or $h_i$ respectively.}

The cell structure on $S^3$, pictured in Figure~\ref{coneonlink.fig}, consists of:
\begin{enumerate}
	\item One $0$-cell, which is the cone point of the cone on the link $\alpha\cup\gamma\cup\delta$.
	\item One ``horizontal''$1$-cell for each arc in the link diagram: these are the $k_i$, $g_i$, and $h_i$.
	\item One ``vertical'' $1$-cell for each crossing in the link diagram.  The vertical 1-cell connecting the head of an arc of $k_i$ or $g_i$ to the $0$-cell is denoted $a_i$ or $b_i$, respectively.
	\item One ``vertical'' $2$-cell for each crossing in the link diagram.  The vertical 2-cell below an arc $k_i$ or  $g_i$ is denoted $A_i$ or $B_i$, respectively.
	\item One 3-cell, $e^3$, which is the complement of the cone on the link.
\end{enumerate}

Note that $\partial A_i=k_i+a_i-a_{i-1}$, and $\partial B_i=g_i+b_i-b_{i-1}$.

Denote by $c(i)$ the color, 1, 2, or 3, assigned to the arc $k_i$.  Let $f(i)$ denote the subscript $j$ of the arc $k_j$ or $g_j$ which passes over crossing $i$ of $\alpha$, and let $f_\gamma(i)$ denote the subscript $j$ of the arc ($k_j$ or $g_j$) passing over crossing $i$ of $\gamma$; $f_\delta(i)$ is defined similarly.   For example, in Figure~\ref{coneonlink.fig}, $f(3)=7$, $f_\gamma(0)=5$, $f_\gamma(1)=6$, and $f_\delta(0)=4$. We will sometimes write $f(i)$ rather than $f_\gamma(i)$ or $f_\delta(i)$ to simplify notation, when it is clear that the under-arc is an arc of $\gamma$ or $\delta$ rather than one of $\alpha$. 

The lists of over-strand subscripts $(f(0),\dots, f(m-1))$ and $(f_\gamma(0),\dots ,f_\gamma(s-1))$ for $\alpha$ and $\gamma$, the list of colors $(c(0),\dots, c(m-1))$ of the arcs of $\alpha$, and two lists containing the signs of crossings (local writhe numbers) for $\alpha$ and $\gamma$, serve as the necessary input to the algorithm.  At this point, the reader may also wish to glance at the Appendix for examples of this input. Examples are worked out in detail in Section ~\ref{ex} (see also Figures \ref{intersectingcurve2.fig} and \ref{intersectingcurve.fig}).  In the figures, the arcs $k_0$ of $\alpha$ and $g_0$ of $\gamma$ are marked with a zero (as is the zeroth arc of $\delta$).  In order to avoid clutter in the figures, we have labelled the only the arcs $k_0,\dots,k_{13}$ of $\alpha$.  We write $i$ instead of $k_i$, and refer to this as a {\it numbering} of the diagram.   The arcs of $\gamma$ should be numbered in a similar fashion.  Note that we ignore the second pseudo-branch curve $\delta$ when numbering the arcs of $\alpha$ and $\gamma$ in the diagram.

\subsection{The cell structure on $M$} \label{liftcells.sec} Now we describe how to lift our cell structure to $M$ and introduce notation for the lifts of the cells.  Our strategy is to understand the lift of the cell structure on $S^3$ in a neighborhood of each crossing, and label the cells near the lift of each crossing in a systematic way.  For example, Figure~\ref{cellsbelowcrossing.fig} shows the cells near a self-crossing of $\alpha$ in $S^3$.  Figure~\ref{upstairsinhomog.fig} shows one way these cells lift if the crossing is {\it inhomogenous}, that is, the colors on the three arcs are all different.  In contrast, Figure~\ref{upstairshomog.fig} shows one way these cells lift if the crossing is {\it homogeneous}, that is, the three colors on the arcs are the same.  Later in this section we explain how these figures are constructed, what the possible configurations of cells above a crossing are, and how to determine which configuration arises.  We must also analyze the lifts of cells near self-crossings of $\gamma$, and near crossings of $\alpha$ under $\gamma$.  We adopt some of the notation of~\cite{perko1964thesis} for the lifts of cells coming from the knot $\alpha$.  We introduce a new way of visualizing the cell structure which simplifies the task of computing linking numbers between pseudo-branch curves, and generalizes easily to the case where $\alpha$ is Fox $p$-colored for $p\geq 5.$

 Let $\alpha_1$ and $\alpha_2$ denote the index-1 and index-2 branch curves in $M$ of the 3-fold irregular branched covering map $f:M\rightarrow S^3$; note $\alpha_1\cup \alpha_2=f^{-1}(\alpha)$. Each arc $k_i$ of $\alpha$ has two pre-images under the covering map.  Let $k_{1,i}$ denote the index-1 lift of $k_i$ and let $k_{2,i}$ denote the index-2 lift of $k_i$.  Let $A_{1,i}$, $A_{2,i}$ and $A_{3,i}$ denote the three lifts of $A_i$; shortly, we will explain which of these 2-cells is given which label.

First, we introduce notation for the lifts of $e^3.$  This 3-cell has three lifts, $e^3_1$, $e^3_2$, and $e^3_3$.  Recall that the color $c(i)$ on the arc $k_i$ of $\alpha$ corresponds to a transposition in $S^3$, which we denote by $\tau_i$.  We label the cells $e^3_j$ such that the lift of a meridian of $k_i$ beginning in the cell $e^3_j$ has its endpoint in $e^3_{\tau_i(j)}$. Figure~\ref{meridianlift.fig} shows how these cells are configured along the lifts of an arc of $\alpha$, away from any crossings in the link diagram.
\begin{figure}[htbp]\includegraphics[width=6in]{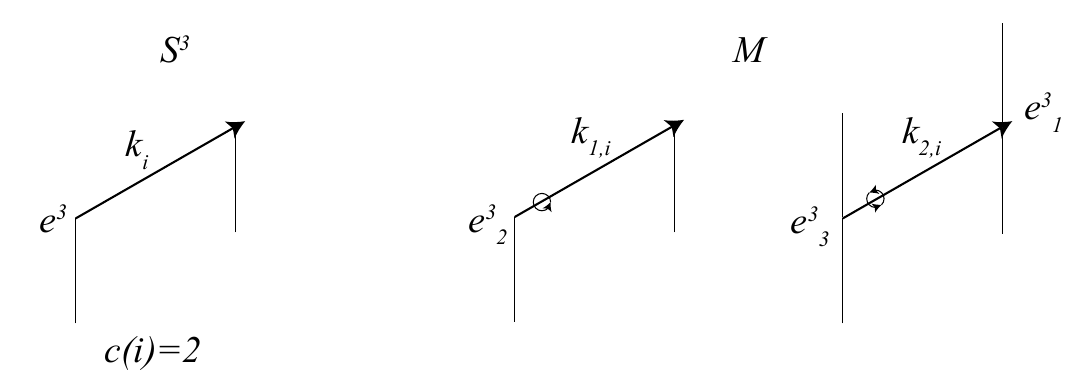}
	\caption{Configuration of the cells $e^3_j$ when the arc $k_i$ is colored 2.}
	\label{meridianlift.fig}
\end{figure}

Now we describe Perko's notation for the lifts of the $A_i$ and the $a_i$, which we also adopt.  For each $i$, one lift of $A_i$ has boundary meeting the index 1 branch curve.  Call this lift $A_{1,i}$.  The other two lifts of $A_i$ share a common boundary segment along the index 2 curve.  These lifts will be called $A_{2,i}$ and $A_{3,i}$.  One makes the choice as follows.  Let $\vec{A}$ be a framing of $\alpha$ tangent to the vertical 2-cells $A_i$.  Now lift $\vec{A}$ to a framing $\vec{A}_{2}$ along the index 2 lift $\alpha_2$ of $\alpha$.  Such a lift exists because the number of crossings in the diagram of $\alpha$ is even.   There are two choices for such a lift.  We make a choice arbitrarily along $k_{2,0}$ and this uniquely determines the lift along the entire curve.  Call $A_{2,i}$ the lift of $A_i$ located in the positive direction of $\vec{A}_2$.  Last, we denote by $a_{j,i}$ the lift of $a_i$ which is a subset of the boundary of $A_{j,i}$ for $j=1,2,3$.  See Figure~\ref{determineA2i.fig}.

\begin{figure}[htbp]
\includegraphics[width=6in]{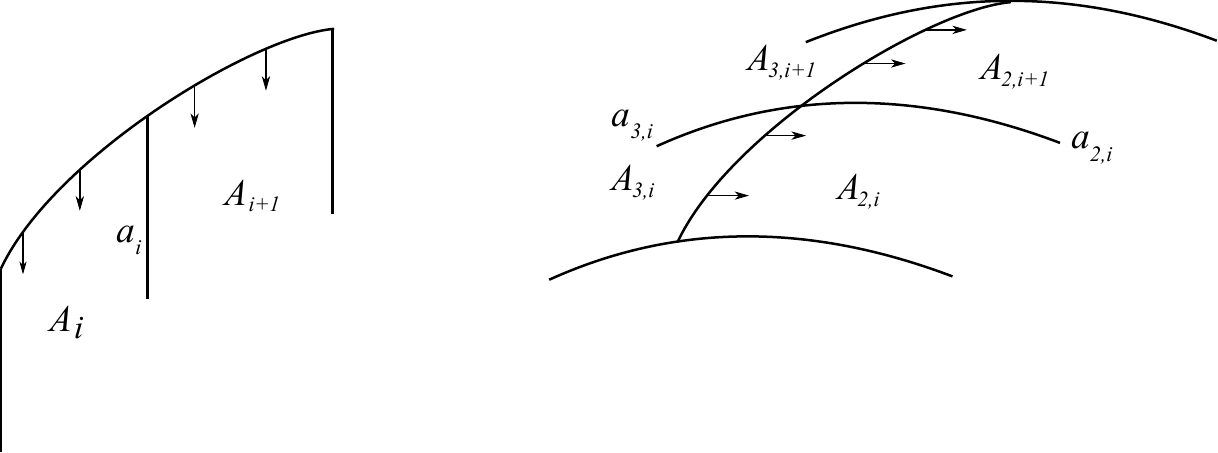}
\caption{A lift of a framing of $\alpha$ along the degree two curve. This lift determines the labelling of the lifts of the 2-cells $A_i$.}
\label{determineA2i.fig}
\end{figure}

\begin{figure}[htbp]
\includegraphics[width=2.5in]{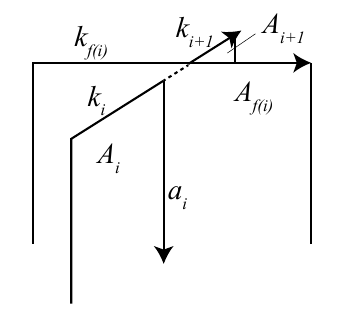}
\caption{Cells at crossing $i$ of $\alpha$.}
\label{cellsbelowcrossing.fig}
\end{figure}
 The next step is to determine how the 2-cells in $M$ are attached to the 1-skeleton; this is essential for finding the required 2-chains for the linking number computation.  There are two cases to consider:  self-crossings of $\alpha$ (either inhomogeneous or homogeneous);  and crossings involving $\gamma$ (self-crossings of $\gamma$, and crossings of $\alpha$ under $\gamma$).

{\bf Case 1: Self-crossings of $ \alpha.$} The cells at a self-crossing of $\alpha$ are shown in Figure~\ref{cellsbelowcrossing.fig}.  We analyze how the lifts of $A_i$, $A_{i+1}$, and $A_{f(i)}$ are assembled. Namely, we need to understand possible configurations of $A_{1,i}$, $A_{2,i}$, $A_{3,i}$, $A_{1,i+1}$, $A_{2,i+1}$, $A_{3,i+1}$, $A_{1,f(i)}$, $A_{2,f(i)}$, and $A_{3,f(i)}$.

{\bf Case 1a: Inhomogeneous self-crossings of $\alpha$.} Figure~\ref{tents.fig} shows one way these cells might lift at an inhomogenous crossing, if $k_i$ is colored `2', $k_{i+1}$ is colored `1', and $k_{f(i)}$ is colored `3'.  Note that in Figure~\ref{tents.fig}, some cells appear twice in the picture---for example, $k_{2,i}$, $A_{2,i}$, and $A_{3,i}$.  We can alternatively visualize these cells as shown in Figure~\ref{upstairsinhomog.fig}; we construct this picture by identifying all duplicate cells in Figure~\ref{tents.fig}.  The positions of $A_{1,i}$ and $A_{1,f(i)}$, relative to the positions of the 3-cells $e^3_j$, are completely determined by this coloring information.  The positions of $A_{2,i}$ and $A_{3,i}$, on the other hand, are determined by global information about the coloring of the knot, rather than just the coloring at that crossing. One possibility is shown in Figure~\ref{tents.fig}, but the position of the 2-cells $A_{2,i}$ and $A_{3,i}$ could be interchanged.  This is also the case for $A_{2,f(i)}$ and $A_{3,f(i)}$. 
\begin{figure}[htbp]
\includegraphics[width=\textwidth]{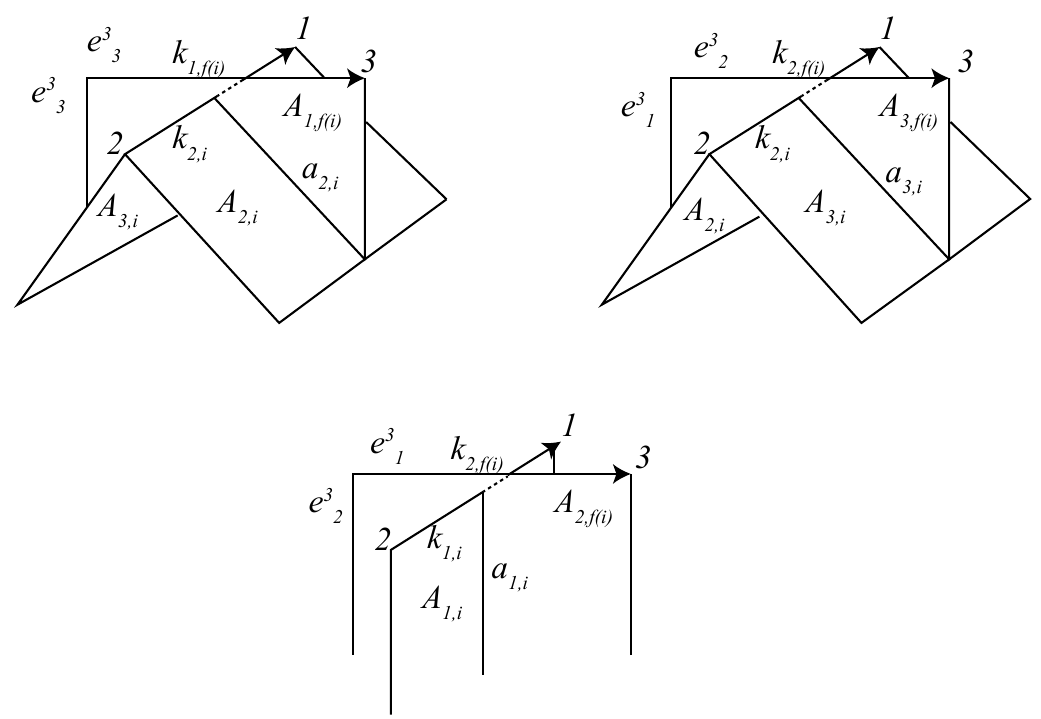}
\caption{One possible configuration of cells above an inhomogeneous crossing~$i$ of $\alpha$.}
\label{tents.fig}
\end{figure}

\begin{figure}[htbp]
\includegraphics[width=5in]{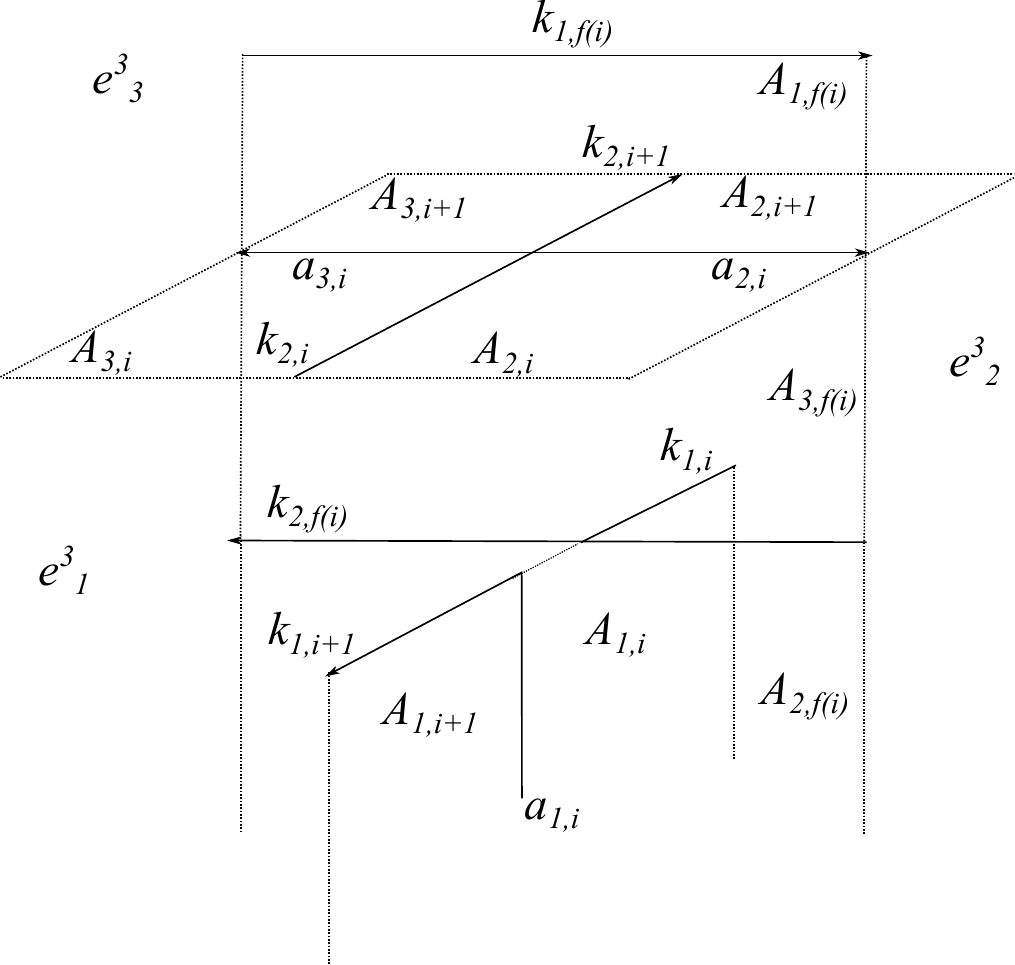}
\caption{One possible configuration of the cells above an inhomogeneous crossing $i$ of $\alpha$. Here, $k_i$ is colored $2$, $k_{i+1}$ is colored $1$, and $k_{f(i)}$ is colored $3$. This picture is obtained by identifying duplicate cells in Figure~\ref{tents.fig}.}
\label{upstairsinhomog.fig}
\end{figure}

Therefore, we need to keep track of the position of $A_{2,i}$ and $A_{3,i}$ relative to the various 3-cells $e^3_j$. To do this, we introduce a function $w(i)$ as follows.  Informally, $w(i)=j$, where $j$ is the subscript of the 3-cell $e^3_j$ such that, if one stands in that 3-cell on the index 2 branch curve $k_{2,i}$ and facing in the direction of its orientation, then $A_{2,i}$ is on the right.  In Figures~\ref{tents.fig} and~\ref{upstairsinhomog.fig}, $w(i)=3$ and $w(f(i))=2$.  

One can easily compute $w(i)$ from $c(i)$ and $f(i)$ as follows:

$$w(i+1)=\left\{\begin{array}{rr}w(i)&\text{if crossing }i\text{ of }\alpha \text{ terminates at an arc of }\gamma \\
\tau_{f(i)}(w(i))&\text{if crossing }i\text{ of }\alpha\text{ terminates at an arc of }\alpha

\end{array}\right.$$
Recall that $\tau_{f(i)}$ denotes the transpostion corresponding to the color $c(f(i))$ on the over-arc at crossing $i$; $\tau_{f(i)}(w(i))$ denotes its action on $w(i)\in \{1,2,3\}.$

There are eight possible configurations of 2-cells above a given inhomogeneous crossing of $\alpha$ with prescribed colors, shown in Figure~\ref{inhomogeneousall.fig}.  In the case of an inhomogeneous crossing, $w(i)$ equals either $c(f(i))$ or $c(i+1)$, and $w(f(i))$ equals either $c(i)$ or $c(i+1)$. We record this information with a pair of functions $\epsilon_1(i)$ and $\epsilon_2(i)$:
\begin{equation}\label{epsilon1}
	\epsilon_1(i)=1 \text{ if } c(i)\neq w(f(i)), \text{ and }\epsilon_1(i)=-1 \text{ if }c(i)=w(f(i))
\end{equation}	

\begin{equation}	\label{epsilon2}
\epsilon_2(i)=1 \text{ if } c(f(i))= w(i), \text{ and }\epsilon_2(i)=-1 \text{ if }c(f(i))\neq w(i).
\end{equation}

 In addition, the crossing may have positive or negative local writhe number $\epsilon(i)$.

	\begin{figure}[htbp]\includegraphics[width=6in]{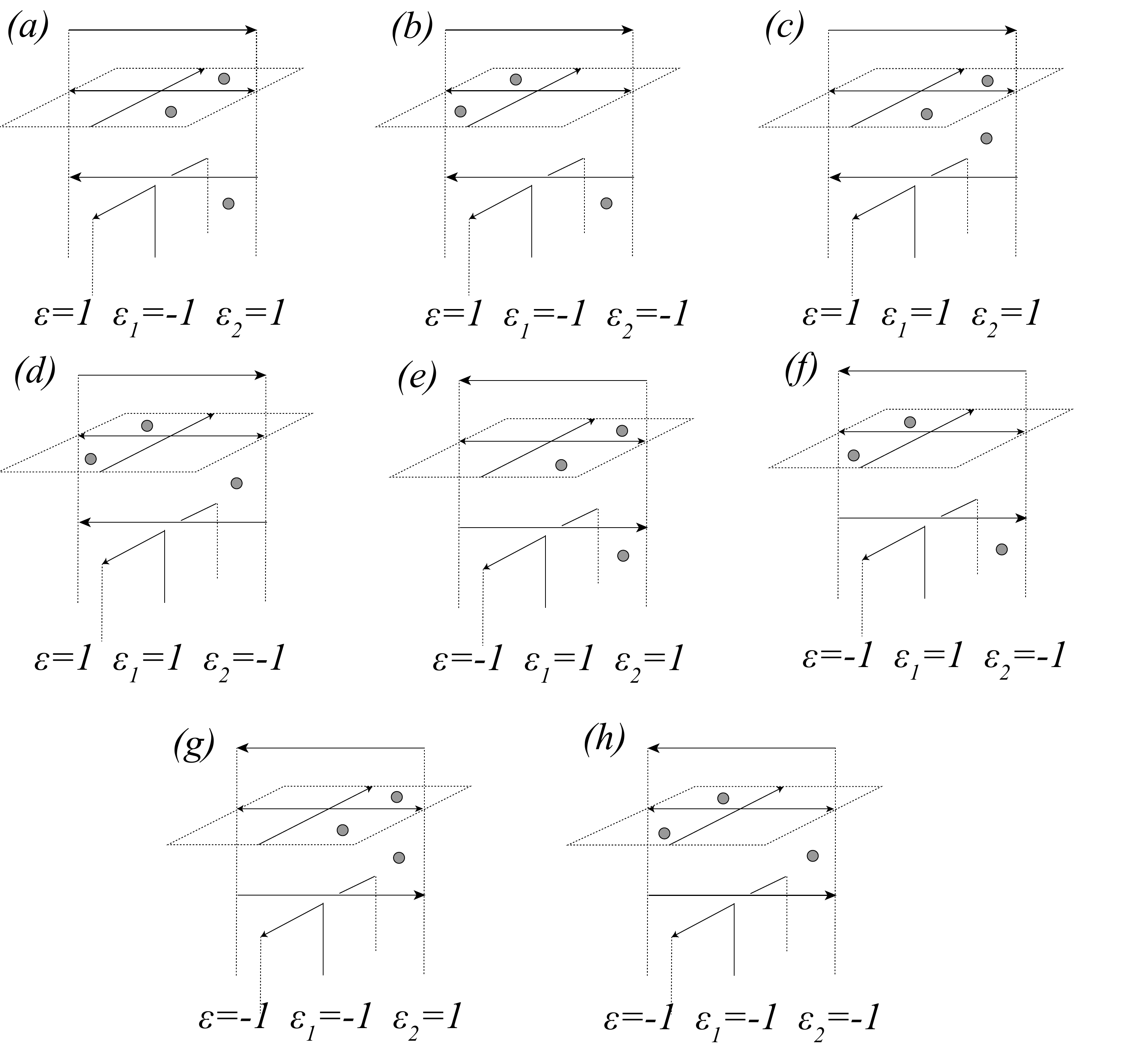}
		\caption{Configurations of cells above an inhomogeneous self-crossing of $\alpha$.  Dotted 2-cells indicate the locations of the cells $A_{2,i}$, $A_{2,i+1}$, and $A_{2,f(i)}$.}\label{inhomogeneousall.fig}
	\end{figure}

{\bf Case 1b: Homogeneous self-crossings of $\alpha$.} In the case of a homogenous crossing of $\alpha$, the colors $c(i), c(i+1)$ and $c(f(i))$ are all equal, and the 3-cell $e^3_{c(i)}$ is adjacent to the arcs $k_{1,i}$, $k_{1,i+1}$, and $k_{1,f(i)}$. See, for example, Figure~\ref{upstairshomog.fig}. There are four possible configurations of 2-cells near the index 2 lift of $\alpha$, shown in Figure~\ref{homogeneousall.fig}; in particular, the value of $w(i)$ either coincides with $w(f(i))$, or not.  We record this information with a function $\epsilon_3(i)$:
\begin{equation} \label{epsilon3}
	\epsilon_3(i)=1 \text{ if } w(i)\neq w(f(i)) \text { and }\epsilon_3(i)=-1 \text{ if } w(i)=w(f(i)).
\end{equation}

\begin{figure}[htbp]
\includegraphics[width=4.5in]{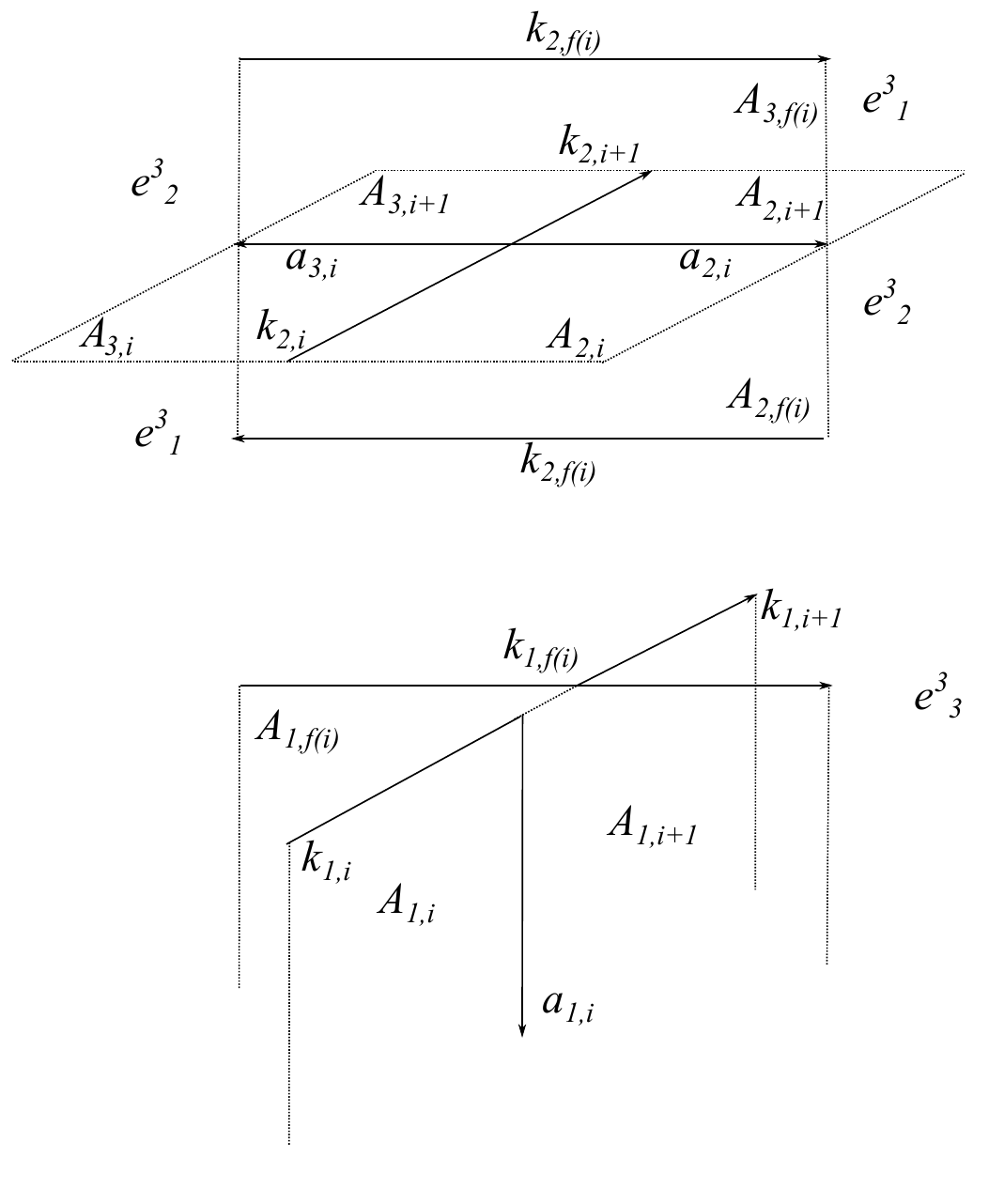}
\caption{One possible configuration of the cells near the lift of a homogeneous positive crossing $i$, with all arcs colored $3$.  The two copies of $k_{2,f(i)}$ are identified. }
\label{upstairshomog.fig}
\end{figure}

	\begin{figure}[htbp]\includegraphics[width=4.5in]{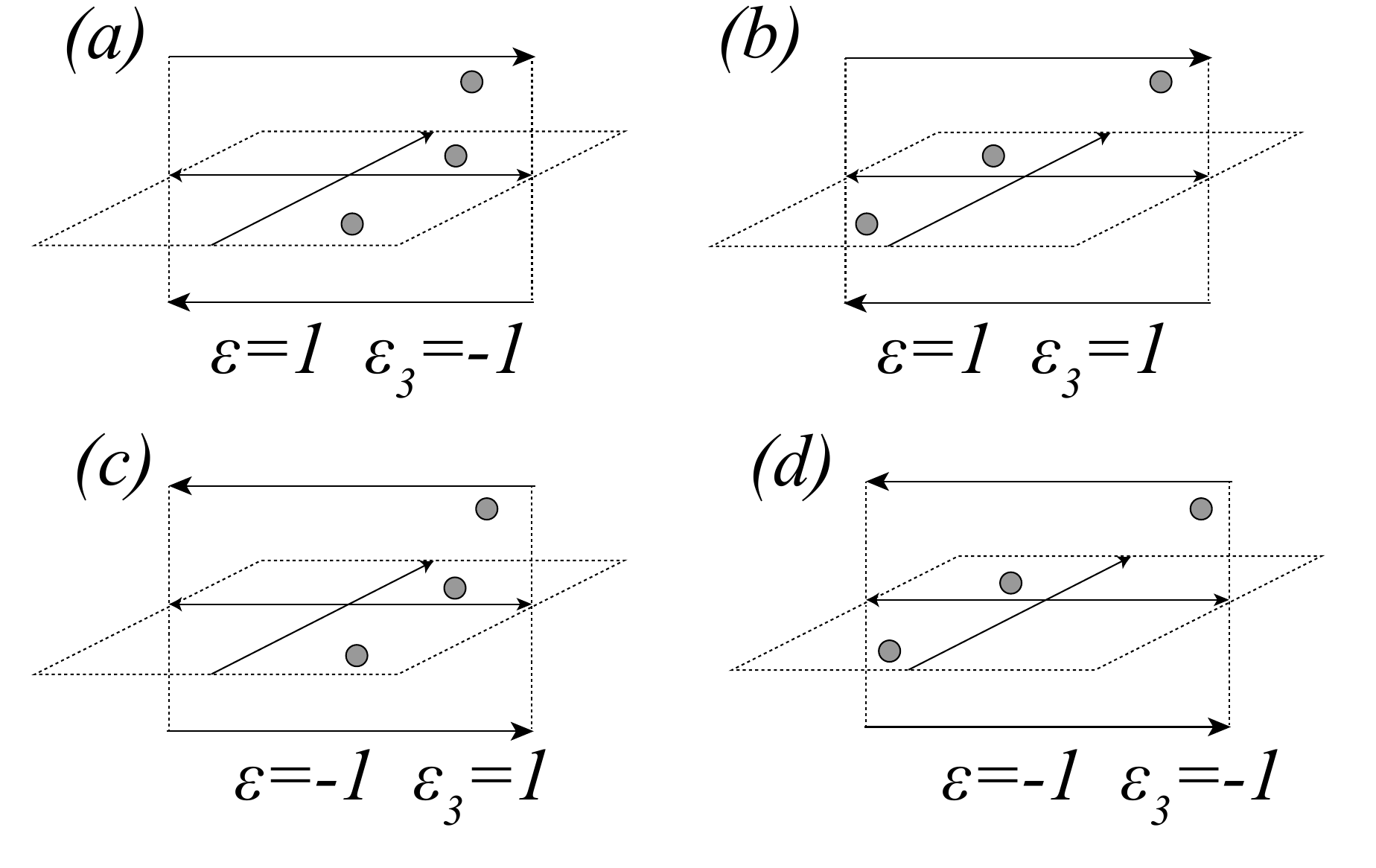}
		\caption{Configurations of cells above a homogeneous self-crossing of $\alpha$.  Dotted 2-cells indicate the locations of the cells $A_{2,i}$, $A_{2,i+1}$, and $A_{2,f(i)}$.}\label{homogeneousall.fig}
	\end{figure}

{\bf Case 2: Crossings involving $\gamma$.} We have now discussed the lifts of all cells in the cone on $\alpha$. At this stage, we introduce notation for the cells in the cone on $\gamma$, which have not played a role so far.

Choose a basepoint $x_0$ on the arc $g_0$ of $\gamma$.  The curve $\gamma$ has three path-lifts under the covering map, $\gamma_1$, $\gamma_2$, and $\gamma_3$, beginning at each of the three preimages of $x_0$.  Assume the $\gamma_i$ are labelled so that the lift of $g_0$ which lies in the 3-cell $e^3_i$ is contained in $\gamma_i$. The pre-image $f^{-1}(\gamma)$ is the union of the lifts $\gamma_1$, $\gamma_2$, and $\gamma_3$, and has one, two or three connected components in $M$.   Let $g_{j,i}$, $j=1,2,3$, denote the lift of $g_i$ which lies in the lift $\gamma_j$ of $\gamma$.   Denote by $B_{j,i}$ the lift of $B_i$ whose boundary contains $g_{j,i}$. 

First we consider self-crossings of $\gamma$. In this case, covering map is locally trivial in a neighborhood of the crossing.  As before, different configurations of 2-cells arise above a self-crossing of $\gamma$; see Figure \ref{upstairspseudooverpseudo.fig} for one example.  We introduce an auxiliary function $l^g_j(i)$, whose value is the subscript $s$ of the 3-cell $e^3_s$ which contains the lift $g_{j,i}$ of the arc $g_i$.  For example, in Figure \ref{upstairspseudooverpseudo.fig}, $l^g_1(i)=2$, $l^g_2(i)=3$, and $l^g_3(i)=1$.

\begin{figure}[htbp]
\includegraphics[width=6in]{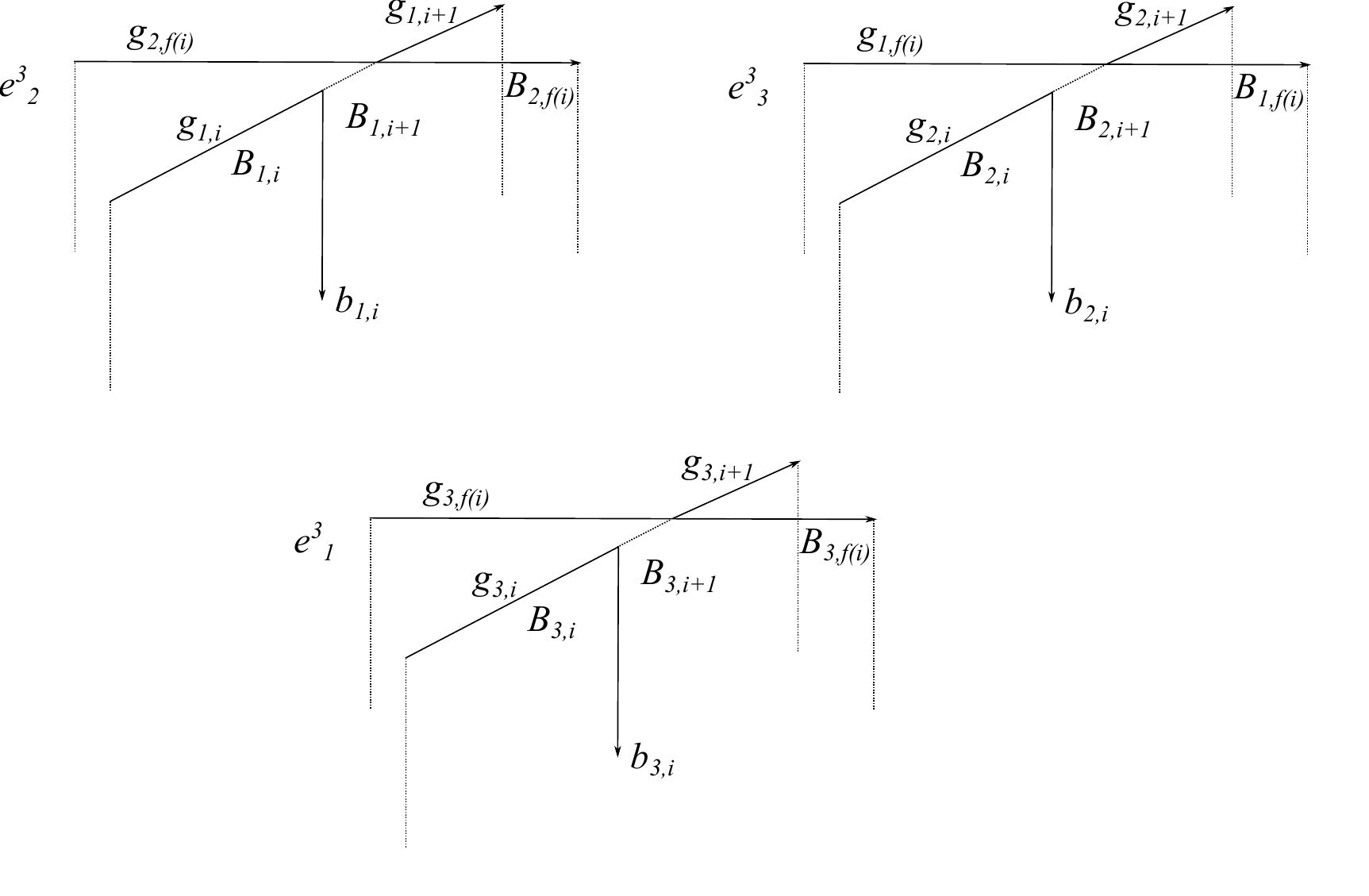}
\caption{One possible configuration of cells lying near the lift of a crossing of $\gamma$ with itself.}
\label{upstairspseudooverpseudo.fig}
\end{figure}

 Next we consider crossings where $\alpha$ passes under the pseudo-branch curve $\gamma$. As in the case of crossings of $\alpha$ with itself, the configuration of cells above that crossing will depend on the value of $w(i)$.   One such configuration is pictured in Figure \ref{upstairsknotunderpseudo.fig}.  All six configurations are shown in Figure~\ref{alphaundergammaall.fig}.  To capture the combinatorics at play, we associate a function to crossings of $\alpha$ under $\gamma$ as follows: 

\begin{equation}\label{epsilon4}
	\epsilon_4^j(i)= \left\{\begin{array}{ll}
 {1}\text{ if } l^g_j({f(i)})=w(i); \\
0 \text{ if }l^g_j({f(i)})=c(i);\\
{-1} \text{ otherwise.}
 \end{array}
\right.
\end{equation}

For example, in Figure \ref{upstairsknotunderpseudo.fig}, $\epsilon^1_4(i)={1}$, $\epsilon^2_4(i)={0}$, and $\epsilon^3_4(i)={-1}$.
\begin{figure}[htbp]
\includegraphics[width=4in]{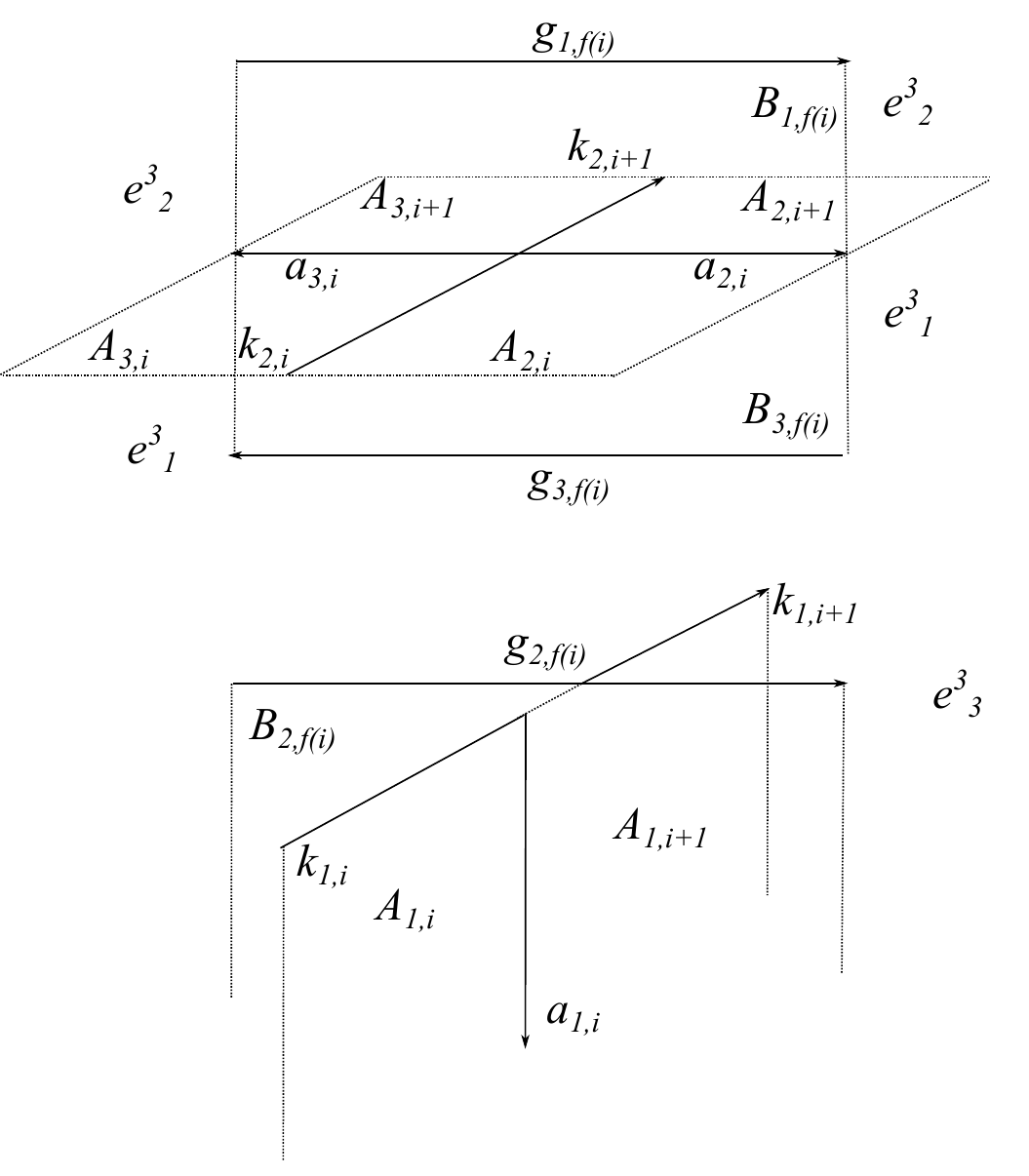}
\caption{One possible configuration of cells near the lift of a crossing where $\gamma$ passes over~$\alpha$. Here, the arc $k_i$ is colored $3$, which determines the subscripts on the three-cells in the picture.} 
\label{upstairsknotunderpseudo.fig}
\end{figure}

	\begin{figure}[htbp]\includegraphics[width=6in]{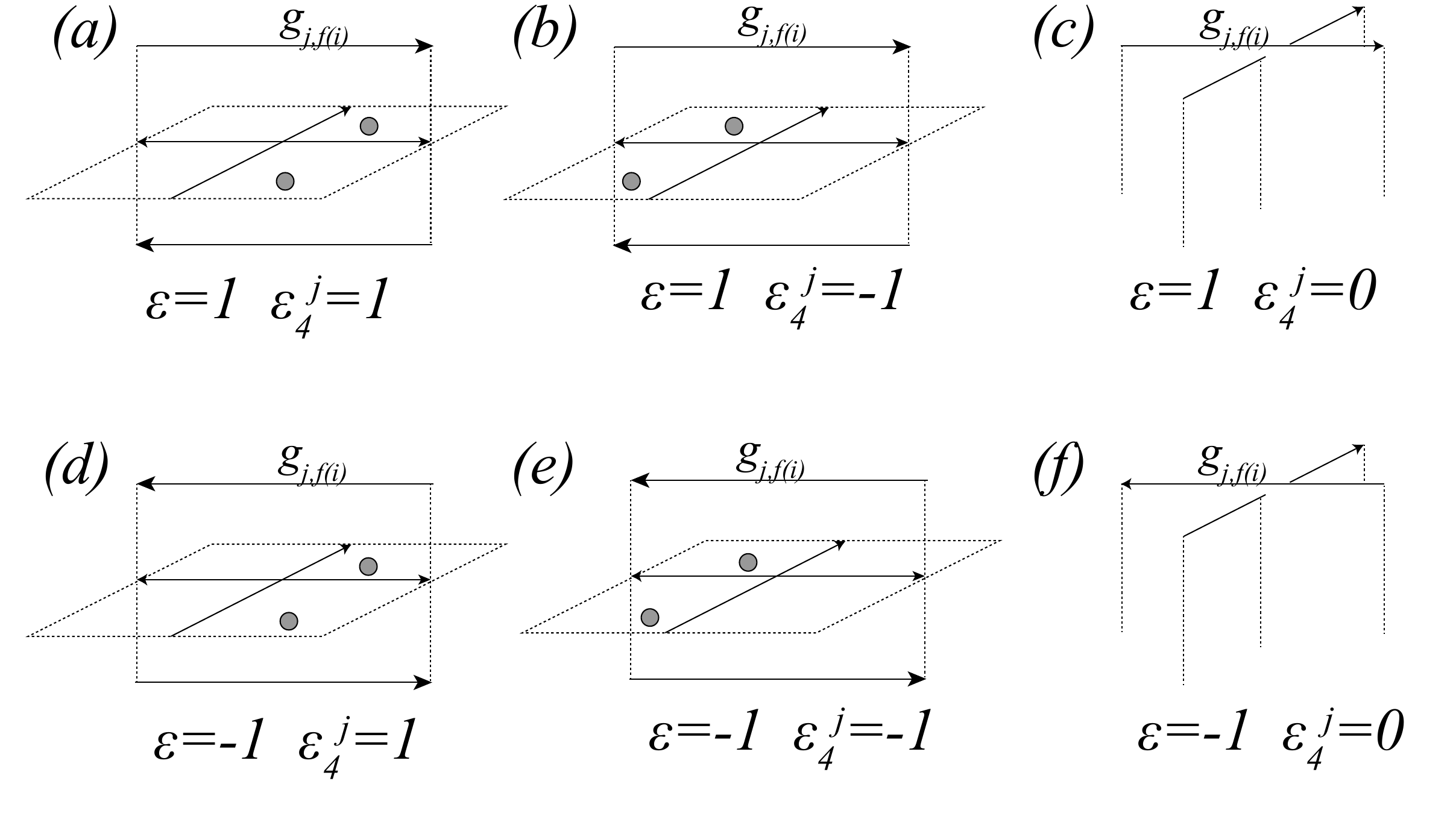}
		\caption{Configurations of cells above a crossing of $\alpha$ under $\gamma$.  Dotted 2-cells indicate the locations of the cells $A_{2,i}$ and $A_{2,i+1}$.}\label{alphaundergammaall.fig}
	\end{figure}

\section{Constructing 2-chains bounding pseudo-branch curves}\label{chains.sec}

\label{pb-surfaces} 
 Our task is to compute the linking numbers between any two lifts of pseudo-branch curves, whenever these linking numbers are well-defined.  In order to compute the linking numbers of pseudo-branch curves, we must find 2-chains bounding them, or determine that no such 2-chains exist.  

For now we assume that the lift of $\gamma$ has three connected components, $
\gamma_1$, $\gamma_2$, and $\gamma_3$.

We look for a 2-chain $C_j$ with  $\partial C_j=\gamma_j$ for fixed $j$.  A priori we have $$C_j=\sum_{i=0}^{s-1} z^j_iB_{j,i}+\sum_{i=0}^{m-1} x^j_i A_{2,i}+y^j_iA_{3,i}.$$

Since $\gamma_j=\sum_{i=0}^{s-1} g_{j,i},$ each 1-cell $g_{j,i}$ must appear exactly once in the boundary of $C_j$; no other 1-cells appear.  Hence $z^j_i=1$ and $y^j_i=-x^j_i.$ Now

\begin{equation}
\label{C2j} C_j=\sum_{i=0}^{s-1} B_{j,i}+\sum_{i=0}^{m-1} x^j_i (A_{2,i}-A_{3,i}).
\end{equation}

It remains to find the coefficients $x^j_i$. To that end, we write down a system of linear equations in the $x^j_i$, one for each crossing.  We obtain three systems of equations, one for each  $C_j$ with $j\in \{1,2,3\}$, given in Proposition~\ref{system-eqns2}.

\begin{customprop}{\ref{system-eqns2}} {\it
 Let $s$ denote the number of crossings of $\gamma$ under $\alpha$ plus the number of self-crossings of $\gamma$, let $m$ denote the number of crossings of $\alpha$ under $\gamma$ plus the number $n$ of self-crossings of $\alpha$.  Let $f(i)$ denote the index of the overstrand $k_{f(i)}$ at crossing $i$, and let the signs $\epsilon$, and $\epsilon_k$ for $k=1,2,3,4$ be as in Table~\ref{notation.tab}.  If the following inhomogeneous system of linear equations 

 $$\left\{\begin{array}{ll}
 x_i^j-x_{i+1}^j+\epsilon_1(i)\epsilon_2(i)x_{f(i)}^j=0 &\text{ if crossing i of } \alpha \text{ is inhomogeneous}\\
 x_i^j-x_{i+1}^j+2\epsilon_3(i)x_{f(i)}^j=0&\text{ if crossing i of }\alpha \text{ is homogeneous}\\
 x^j_i-x^j_{i+1}=\epsilon(i)\epsilon^j_4(i)&\text{if strand i of }\alpha \text{ passes under }\gamma
 \end{array}
 \right.$$

 has a solution $(x_0^j,x_1^j,\dots,x_{m-1}^j)$ over $\mathbb{Q}$, then the lift $\gamma_j$ of $\gamma$ is rationally null-homologous and is bounded by the 2-chain
 $$C_j=\sum_{i=0}^{s-1} B_{j,i}+\sum_{i=0}^{m-1} x_i^j(A_{2,i}-A_{3,i}).$$  }

\end{customprop}
\begin{proof}

	Our goal is to find the coefficients $x^j_i$ in the 2-chain $C_j$ above.  We take advantage of the fact that the lifts of the 1-cell $a_i$ appear only above crossing $i$ of $\alpha$; this may be a crossing of $\alpha$ under $\alpha$, or a crossing of $\alpha$ under $\gamma$.  We then compute the contribution of lifts of $a_i$ to $\partial C_j$ at three types of crossings: inhomogeneous crossings of $\alpha$, homogeneous crossings of $\alpha$, and crossings of $\alpha$ under $\gamma$. Our system of linear equations is obtained by setting each of these contributions to zero.

	Consider the eight possible configurations of 2-cells above an inhomogeneous crossing of $\alpha$, shown in Figure~\ref{inhomogeneousall.fig}.  The ``vertical" 1-cells $a_{2,i}$ and $a_{3,i}$ appear in $\partial C_j$ in pairs with opposite sign.  We compute the number of times the 1-chain $a_{2,i}-a_{3,i}$ appears in $$\partial(\sum_{i=0}^{s-1} B_{j,i}+\sum_{i=0}^{m-1} x_i^j(A_{2,i}-A_{3,i}))$$ for each configuration, and set this equal to zero.
	
	 We get the following eight equations, corresponding to each of the eight configurations.  
		\begin{multicols}{3}\begin{enumerate}[label={\it (\alph*)}]
			\item $x^j_i-x^j_{i+1}-x^j_{f(i)}=0$
	\item $x^j_i-x^j_{i+1}+x^j_{f(i)}=0$
	\item $x^j_i-x^j_{i+1}+x^j_{f(i)}=0$
	\item $x^j_i-x^j_{i+1}-x^j_{f(i)}=0$
	\item $x^j_i-x^j_{i+1}+x^j_{f(i)}=0$
	\item $x^j_i-x^j_{i+1}-x^j_{f(i)}=0$
	\item $x^j_i-x^j_{i+1}-x^j_{f(i)}=0$
	\item $x^j_i-x^j_{i+1}+x^j_{f(i)}=0$
			\end{enumerate}
		
		\end{multicols}
	Following~\cite{perko1964thesis}, we may rewrite the eight equations above in terms of $\epsilon_1$ and $\epsilon_2$ to consolidate them into just one equation:

\begin{equation}
\label{inhomo} 
x_i^j-x_{i+1}^j+\epsilon_1(i)\epsilon_2(i)x_{f(i)}^j=0.
\end{equation}

	Similarly, for homogeneous self-crossings of $\alpha$ we have the following equations, corresponding to the four possible configurations in Figure~\ref{homogeneousall.fig}. 
		\begin{multicols}{2}\begin{enumerate}[label={\it (\alph*)}]
			\item $x^j_i-x^j_{i+1}-2x^j_{f(i)}=0$
			\item $x^j_i-x^j_{i+1}+2x^j_{f(i)}=0$
			\item $x^j_i-x^j_{i+1}+2x^j_{f(i)}=0$
			\item $x^j_i-x^j_{i+1}-2x^j_{f(i)}=0$
			\end{enumerate}
		
		\end{multicols}
We can again consolidate them into one equation, this time using $\epsilon_3$:

\begin{equation}
\label{homo} 
x_i^j-x_{i+1}^j+2\epsilon_3(i)x_{f(i)}^j=0.
\end{equation}

Now we consider crossings of $\alpha$ under $\gamma$, as in Figure \ref{upstairsknotunderpseudo.fig}.  There are six possibile configurations for the 2-cells above crossings of $\alpha$ under $\gamma$, shown in Figure~\ref{alphaundergammaall.fig}.   We again count the number of times the 1-chain $a_{2,i}-a_{3,i}$ appears in each and set this equal to zero.  The corresponding equations are shown below. 

		\begin{multicols}{3}\begin{enumerate}[label={\it (\alph*)}]
			
	\item $x^j_i-x^j_{i+1}=1$
	\item $x^j_i-x^j_{i+1}=-1$
	\item $x^j_i-x^j_{i+1}=0$
	\item $x^j_i-x^j_{i+1}=-1$
	\item $x^j_i-x^j_{i+1}=1$
	\item $x^j_i-x^j_{i+1}=0$
			\end{enumerate}
		
		\end{multicols}
		
 Rewriting in terms of $\epsilon$ and $\epsilon_4^j$ gives

\begin{equation}
\label{knot-under-pb} 
x^j_i-x^j_{i+1}=\epsilon(i)\epsilon^j_4(i).
\end{equation}

Unlike the previous two, this equation does depend on $j$; the right hand side will be $1$ for one lift, $-1$ for another, and $0$ for the third.

The boundary of $C_j$ is then, by construction, $\sum_{i=0}^{s-1} g_{j,i}= \gamma_j$. 
\end{proof}

\section{Computing linking numbers and proof of Theorem~\ref{main-thm}}
\label{epsilons}

To complete the computation, we introduce the second pseudo-branch curve $\delta$ into the diagram $\alpha\cup\gamma$ {\it without changing the subscripts} on the arcs $k_i$ of $\alpha$ or the arcs $g_i$ of $\gamma$.  We label the arcs of $\delta$ by $h_0,\dots, h_{t-1}$, where $t$ is the number of crossings of $\delta$ under $\alpha$ plus the number of crossings of $\delta$ under $\gamma$.  (Self-crossings of $\delta$ do not contribute anything to the linking number.  When numbering arcs of $\delta$ for the computer program, we will assign consecutive arcs of $\delta$ the same number if they are separated by an overcrossing by another arc of $\delta$, in order to slightly simplify the input.)  

We again use the notation $f_\delta(i)$, or just $f(i)$, to denote the subscript of the overstrand at the head of the arc $h_i$.  As was the case with $\gamma$, the preimage of the curve $\delta$ may have one, two or three connected components.  We begin with the case where the preimages of both  $\gamma$ and $\delta$  are three closed loops.  Let $\delta_1$, $\delta_2$, and $\delta_3$ denote the three lifts of $\delta$; as before, we choose the subscripts on the $\delta_k$ so that the lift of $h_0$ which is contained in the 3-cell $e^3_k$ is a subset of $\delta_k$.  Let $h_{k,i}$ denote the lift of $h_i$ which is a subset of $\delta_k$.  Let $l^h_k(i)$ denote the subscript $s$ of the 3-cell $e^3_s$ which contains the arc $h_{k,i}$. We now compute the linking number $I_{j,k}$ of $\gamma_j$  with $\delta_k$, which amounts to proving our main theorem.

\begin{customthm}{\ref{main-thm}}
{\it Let $f: M\to S^3$ be a three-fold irregular dihedral cover branched along a knot $\alpha$, and let $\gamma, \delta\subset S^3-\alpha$.  If the lifts $\gamma_j$ and $\delta_k$ are rationally null-homologous closed loops in $M$ for $j,k\in\{1,2,3\}$, then the linking number $I_{j,k}$ of  $\gamma_j$ with $\delta_k$ is the sum:

$$I_{j,k}=\sum_{i=0}^{t-1} c_i,$$
where $c_i$ is given by

$$c_i =\left\{\begin{array}{ll}
\epsilon_5^k(i) x^j_{f(i)}&\text{if }h_i \text{ terminates at an arc } k_{f(i)} \text{ of }\alpha;\\
\epsilon_\delta(i)\epsilon_6^{j,k}(i)&\text{if }h_i \text{ terminates at an arc of }\gamma;\\
0&\text{if }h_i\text{ terminates at an arc of }\delta.\\
\end{array}
\right.$$}
      \end{customthm}

\begin{proof}[Proof of Theorem~\ref{main-thm}]
Assume that we have found a solution $(x^j_0,\dots,x^j_{m-1})\in \mathbb{Q}^m$ to the set of equations in Proposition~\ref{system-eqns2}.  Then the 2-chain bounding $\gamma_j$ is 
$$\sum_{i=0}^{s-1} B_{j,i}+\sum_{i=0}^{m-1} x^j_i( A_{2,i}- A_{3,i}).$$

Crossings of $\delta$ under both $\alpha$ and $\gamma$ may contribute to the linking number.  Self-crossings of $\delta$ do not contribute to the linking number, which is why our numbering system ignores these crossings. One possible configuration of cells above a crossing of $\delta$ under {$\alpha$} is shown in Figure \ref{upstairspseudo2underknot.fig}. {A schematic showing all possible configurations is shown in Figure~\ref{deltaunderalphaall.fig}.}  The lift $h_{k,i}$ will intersect one of the cells $A_{1,f(i)}$, $A_{2,f(i)}$ or $A_{3,f(i)}$.  If it intersects $A_{1,f(i)}$, this crossing does not contribute to $I_{j,k}$ because $A_{1,f(i)}$ is never contained in the 2-chain bounding $\gamma_j$.  If it intersects $A_{2,f(i)}$, the crossing contributes $\epsilon_\delta(i)x^j_{f(i)}$ to $I_{j,k}$.  If it intersects $A_{3,f(i)}$, the crossing contributes $-\epsilon_\delta(i)x^j_{f(i)}$ to $I_{j,k}$.  

We now work out this contribution for each of the six configurations in Figure~\ref{deltaunderalphaall.fig}.

\begin{multicols}{3}
	\begin{enumerate}[label={\it (\alph*)}]	
	\item $+x^j_{f(i)}$
	\item $-x^j_{f(i)}$
	\item $+0$
	\item $-x^j_{f(i)}$
	\item $+x^j_{f(i)}$	
		\item $+0$
\end{enumerate}
\end{multicols}

We define $\epsilon_5^k$ as follows:

\begin{equation}\label{epsilon5}
	\epsilon_5^{k}(i)= \left\{\begin{array}{ll}
1\text{ if } l^h_k(i)=w({f(i)}) \text{, }\\
0 \text{ if }l^h_k(i)=c({f(i)})\text{, and }\\
-1 \text{ otherwise.}
\end{array}
\right.
\end{equation}

The contribution to $I_{j,k}$ of a crossing of $\delta$ under $\alpha$ is then {$\epsilon_5^k(i)x^j_{f(i)}$}.
\begin{figure}[htbp]
\includegraphics[width=4in]{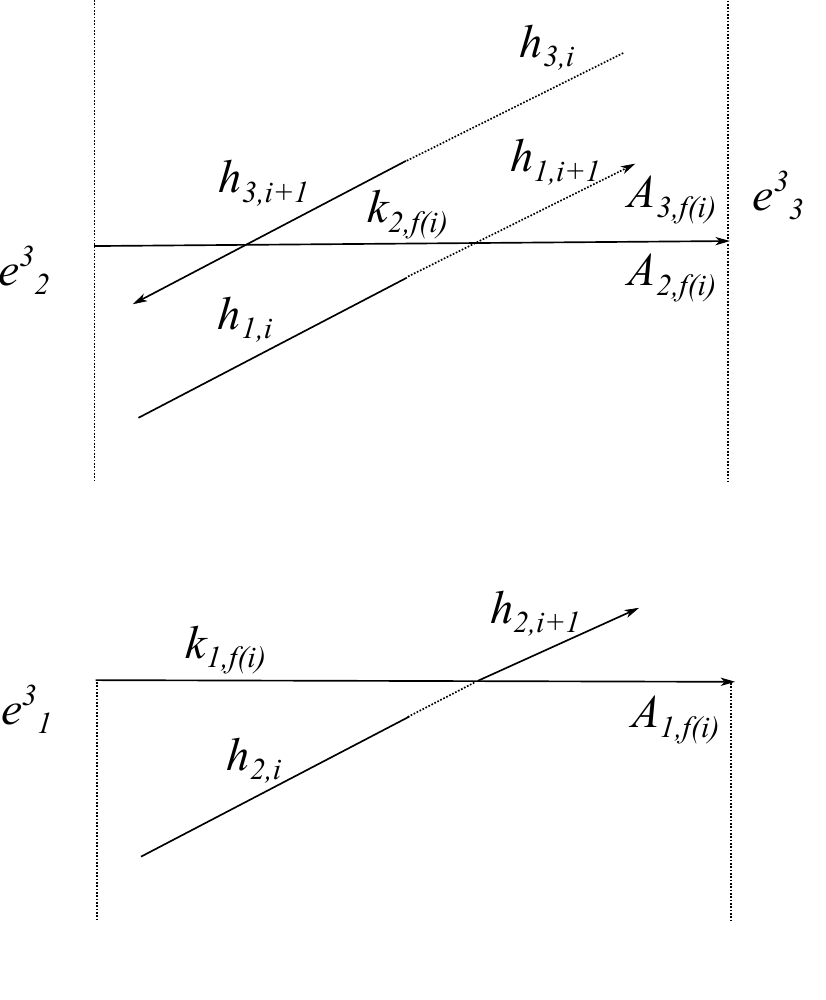}
\caption{One possible configuration of cells near the lift of a crossing where $\delta$ passes under $\alpha$.}
\label{upstairspseudo2underknot.fig}
\end{figure}

	\begin{figure}[htbp]\includegraphics[width=6.5in]{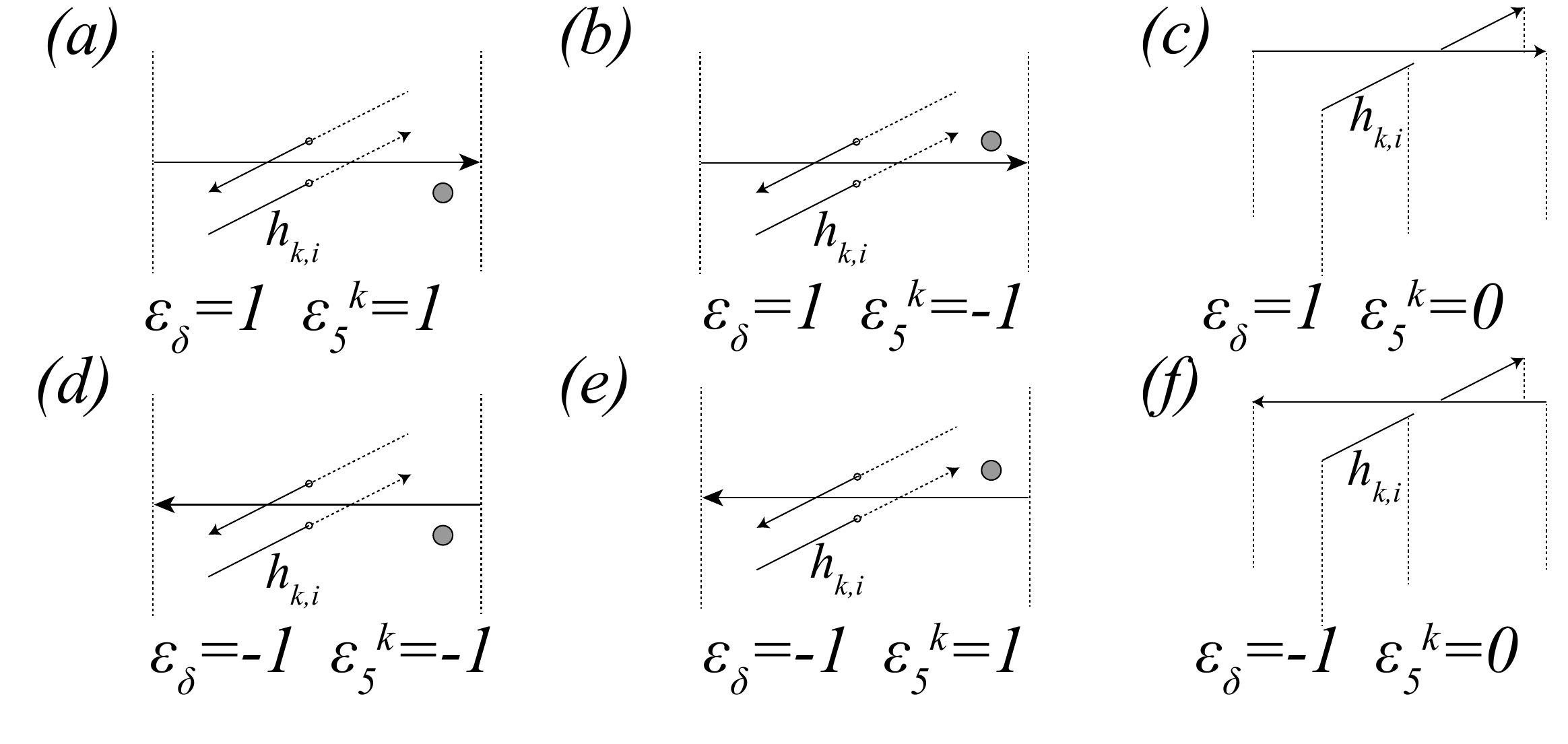}
		\caption{Configurations of cells above a crossing of $\delta$ under $\alpha$.  Dotted 2-cells indicate the locations of the cells $A_{2,i}$ and $A_{2,i+1}$.}\label{deltaunderalphaall.fig}
	\end{figure}
	
Now consider crossings of $\delta$ under $\gamma$.  The picture in the cover is similar to that of Figure~\ref{upstairspseudooverpseudo.fig}, except that the under-crossing arcs are $h_{\cdot,i}$'s rather than $g_{\cdot,i}$'s.  The cell $B_{j,f(i)}$ appears in the 2-chain bounding $\gamma_j$ exactly once, so the contribution of such a crossing to $I_{j,k}$ is $\epsilon_\delta(i)$ if the lifts of $h_{k,i}$ and $g_{j,f(i)}$ are in the same 3-cell, and $0$ otherwise.

Define $\epsilon_6$ as follows:
\begin{equation}\label{epsilon6}
	\epsilon_6^{j,k}(i)= 
\left\{\begin{array}{ll}
1 \text{ if } l^h_k(i)=l^g_j(f(i)), \text{ and }\\
0 \text { otherwise.}
\end{array}
\right.
\end{equation}

By construction,  crossings of $\delta$ under $\gamma$ contribute {$\epsilon_\delta(i)\epsilon^{j,k}_6(i)$} to $I_{j,k}$. The theorem follows.\end{proof}

\subsection{A note on pseudo-branch curves which lift to fewer than 3 loops} 
\label{fewer-lifts}
The pre-image of a pseudo-branch curve $\gamma$ under the covering map may well have fewer than three connected components. Precisely, the lifts of $\gamma$ could include two closed loops $\gamma_1\cdot\gamma_2$ and $\gamma_3$, or one closed loop $\gamma_1\cdot\gamma_2\cdot\gamma_3$, where each $\gamma_j$ covers $\gamma$ and $\cdot$ denotes concatenation of paths.

If some concatenation $\sigma$ of the $\gamma_i$'s forms a closed, rationally null-homologous loop, we can still find a 2-chain $C_{\sigma}$ with boundary $\sigma$ using the methods given in the previous Section~\ref{pb-surfaces} . We do this by writing down the three systems of equations for $j=1,2,3$ listed in Proposition~\ref{thm2}.  The 2-chain $C_{\sigma}$ bounding $\sigma$ is then:

$$C_{\sigma}=\sum_{j\in S}
\left(\sum_{i=0}^{s-1} B_{j,i}+\sum_{i=0}^{m-1} x_i^j(A_{2,i}-A_{3,i})\right).$$

Now let's consider the linking number between two such pseudo-branch curves. Suppose the closed loop $\sigma$ is a concatenation of paths $\gamma_i$, where $i\in S\subset \{1,2,3\}$, and the closed loop $\tau$ is a concatenation of paths $\delta_i$, where $i\in T\subset\{1,2,3\}$ and each $\delta_i$ is a lift of a second pseudo-branch curve $\delta\subset S^3-\alpha$. It follows from Section~\ref{fewer-lifts} that, in the notation of the same section, if $\sigma$ and $\tau$ are rationally null-homologous, their linking number is equal to $\sum_{j\in S, k\in T}I_{j,k}.$

\subsection{Linking numbers between branch curves}  \label{brlinking.sec}

Proposition~\ref{system-eqns2} and Theorem~\ref{main-thm} also allow us to compute linking number of the branch cuves $\alpha_1$ and $\alpha_2$, as follows.

Let the pseudo-branch curve $\gamma$ be a push-off of $\alpha$ along the vector field $\vec{A}$ from Section~\ref{liftcells.sec}. Since the diagram of $\alpha$ has an even number of crossings, $\gamma$ has three lifts.  Two are isotopic to the index-2 lift of $\alpha$ (these are push-offs of $\alpha_2$ along $\pm \vec{A}_2$), and one is isotopic to the index-1 lift of $\alpha$.  Now take a second push-off $\delta$ of $\alpha$ along $\vec{A}$, disjoint from $\gamma$.  Theorem~\ref{main-thm} applied to a diagram of the link $\alpha\cup\gamma\cup\delta$ gives the linking number of $\alpha_1$ and $\alpha_2$.

From this point of view, our results generalize the result of Perko \cite{perko1964thesis}, which gives an algorithm for computing the linking number of $\alpha_1$ and $\alpha_2$ using a cell structure determined by the cone on $\alpha$.  Recall the cell structure we introduce in Section~\ref{liftcells.sec} is a subdivision of Perko's cell structure.
\label{linking-k}
 \begin{prop}[Perko~\cite{perko1964thesis}]
 \label{system-eqs1}

	 If the following inhomogeneous system of linear equations

$$\left\{\begin{array}{ll}
x^1_i-x^1_{i+1}+\epsilon_1(i)\epsilon_2(i)x^1_{f(i)}=\epsilon(i)\epsilon_2(i) &\text{ if crossing i is an inhomogeneous crossing of }\alpha\\
x^1_i-x^1_{i+1}+2\epsilon_3(i)x^1_{f(i)}=0&\text{ if crossing i is a homogeneous crossing of }\alpha\\

\end{array}
\right.$$

has a solution $(x^1_0,x^1_1,\dots,x^1_{m-1})$ over $\mathbb{Q}$, then the index-1 branch curve $\alpha_1$ is rationally null-homologous and is bounded by the 2-chain
$$\sum_{i=0}^{m-1} A_{1,i}+x^1_i(A_{2,i}-A_{3,i}).$$

Similarly if the following system
$$\left\{\begin{array}{ll}
x^2_i-x^2_{i+1}+\epsilon_1(i)\epsilon_2(i)x^2_{f(i)}=\frac{\epsilon_2(i)}{2}\left(\epsilon_1(i)-\epsilon(i)\right)&\text{ if crossing i of } \alpha \text{ is inhomogeneous}\\
x_i^2-x^2_{i+1}+2\epsilon_3(i)x^2_{f(i)}=\epsilon_3(i)&\text{ if crossing i of }\alpha \text{ is homogeneous}\\

\end{array}
\right.$$
has a solution $(x^2_0,x^2_1,\dots,x^2_{m-1})$ over $\mathbb{Q}$ then the index-2 branch curve $\alpha_2$ is rationally null-homologous and is bounded by the 2-chain
$$\sum_{i=0}^{m-1} x^2_iA_{2,i}+(1-x^2_i)A_{3,i}.$$
\label{perko.prop}
\end{prop}
\subsection{Linking numbers between branch and pseudo-branch curves} \label{brpblinking.sec}

By again letting $\gamma$ be a push-off of $\alpha$, we can use Proposition~\ref{system-eqns2} and Theorem~\ref{main-thm} to compute the linking number between the lifts of another pseudo-branch curve $\delta$ with the two branch curves, where the branch curves are isotopic to lifts of $\gamma$.  However, this requires using a numbered diagram of the link $\alpha\cup\gamma\cup\delta$.

Alternatively, one can compute the linking numbers of the lifts $\delta_1$, $\delta_2$ and $\delta_3$ of a pseudo-branch curve $\delta$ with the branch curves $\alpha_1$ and $\alpha_2$ using only the diagram $\alpha\cup\delta.$ We use Proposition~\ref{system-eqs1} above, which gives 2-chains bounding $\alpha_1$ and $\alpha_2$ in terms of the cell structure derived from the cone on $\alpha$.

Arcs of the diagram of $\alpha$ are labelled $k_0,\dots,k_m$, where $m$ is now simply the number of self-crossings of $\alpha$; we continue to assume $m$ is even.  Adjacent arcs separated by the overarc $k_{f(i)}$ are labelled $k_i$ and $k_{i+1}$.  As before, we introduce the pseudo-branch curve $\delta$ to the diagram without changing the labelling on the arcs of $\alpha$. The arcs of $\delta$ are labelled $h_0,h_1, \dots h_t$, where $t$ denotes the number of crossings of $\delta$ under $\alpha$.  Adjacent arcs of $\delta$ separated by an overstrand of $\delta$ are given the same label $h_i$ and viewed as one arc, and adjacent arcs of $\delta$ separated by the overstrand $k_{f_\delta(i)}$ of $\alpha$ are labelled $h_i$ and $h_{i+1}$. Now, from this numbered diagram, we compute the linking numbers between branch and pseudo-branch curves by the formula given in Theorem~\ref{pb-b} below.

\begin{thm}\label{pb-b}
Suppose that the pseudo-branch curve $\delta$ lifts to three null-homologous closed loops $\delta_k$ for $k\in\{1,2,3\}$.  Let $\{x^1_i\}$ and $\{x^2_i\}$ be the solutions to the two systems of equations in Proposition~\ref{perko.prop}.  The linking number $I_{k}^1$ of $\delta_k$ with the index 1 branch curve $\alpha_1$ is

$$\sum_{i=0}^t c_i$$
where $c_i$ is given by $\epsilon_5^k(i)x^1_{f(i)}+\epsilon_\delta(i)(1-|\epsilon_5^k(i)|)$.

	The linking number $I_{k}^2$ of $\delta_k$ with the index 2 branch curve $\alpha_2$ is

$$\sum_{i=0}^t c_i$$
where $c_i$ is given by $\epsilon_5^k(i) x^2_{f(i)}+\frac{\epsilon_5^k(i)}{2}(\epsilon_\delta(i)\epsilon_5^k(i)-1)$.

\end{thm}

\begin{proof}
	
	The index 1 curve $\alpha_1$ is the boundary of the 2-chain
	$$\sum_{i=0}^{m-1} A_{1,i}+x^1_i(A_{2,i}-A_{3,i}).$$
	
	We compute the contribution to the linking number of $\delta_k$ with $\alpha_1$ for each crossing of $\delta$ under $\alpha$.  Recall that the possible configurations of cells above a crossing of $\delta$ under $\alpha$ are shown in Figure~\ref{deltaunderalphaall.fig}.  The contribution for each configuration is:
	
	\begin{multicols}{3}
		\begin{enumerate}[label={\it (\alph*)}]
			\item $+x^1_{f(i)}$
			\item $-x^1_{f(i)}$
			\item $+1$
			\item $-x^1_{f(i)}$
			\item $+x^1_{f(i)}$
			\item $-1$
			\end{enumerate}
	\end{multicols}
	
	We rewrite the contributions above in terms of $\epsilon_\delta$ and $\epsilon_5^k$ to get 
	$$c_i=\epsilon_5^k(i)x^1_{f(i)}+\epsilon_\delta(i)(1-|\epsilon_5^k(i)|).$$

	The index 2 curve $\alpha_2$ is the boundary of the 2-chain
	$$\sum_{i=0}^{m-1} x^ 2_i A_{2,i}+(1-x^2_i)A_{3,i}.$$
	The contribution for each configuration is:
	
		\begin{multicols}{3}
			\begin{enumerate}[label={\it (\alph*)}]
				\item $+x^2_{f(i)}$
				\item $+(1-x^2_{f(i)})$
				\item $0$
				\item $-x^2_{f(i)}$
				\item $-(1-x^2_{f(i)})$
				\item $0$
				\end{enumerate}
		\end{multicols}
		
		We rewrite the contributions above in terms of $\epsilon_\delta$ and $\epsilon_5^k$ to get 
		$$c_i=\epsilon_5^k(i) x^2_{f(i)}+\frac{\epsilon_5^k(i)}{2}(\epsilon_\delta(i)\epsilon_5^k(i)-1).$$
\end{proof}
\section{Examples}
\label{ex}

To conclude, we illustrate the output of the algorithm on a collection of pseudo-branch curves. The branch curve $\alpha$ is a connected sum of two 3-colored trefoil knots. Since the trefoil is a 2-bridge knot, its irregular 3-fold dihedral cover is again $S^3$.   From there, one can check that the irregular 3-fold dihedral cover of $S^3$ branched along the connected sum $\alpha$ is $S^1\times S^2$.

Now we choose pseudo-branch curves on which to perform our computations.  We choose curves which appear in our primary applications---see Section~\ref{applications.subsec}.  We briefly explain the context here, though it is not necessary for understanding the linking number computation itself.

\subsection{Characteristic knots} Cappell and Shaneson proved in~\cite{CS1984linking} that the regular and irregular $p$-fold dihedral covers of $(S^3, \alpha)$ can be constructed from a $p$-fold cyclic cover of $S^3$ branched along an associated knot $\beta\subset S^3-\alpha$, which they called a {\it mod $p$ characteristic knot} for $\alpha$.  They also showed that mod $p$ characteristic knots for $\alpha$, up to equivalence, are in one-to-one correspondence with $p$-fold irregular dihedral covers of $\alpha$.  For a precise definition, let $V$ be a Seifert surface for $\alpha$ and $L_V$ the corresponding linking form. A knot $\beta\subset V^\circ$ is a {\it mod $p$ characteristic knot} for $\alpha$ if $[\beta]$ is primitive in $H_1(V; \mathbb{Z})$ and $(L_V+L_V^T)\beta \equiv 0 \mod p$.

The characteristic knots of $\alpha$ play an essential role in many of the potential applications of this work, including the computation of Casson-Gordon invariants \cite{litherland1980formula}, the Rokhlin $\mu$ invariant \cite{CS1984linking}, and the computation of the invariant $\Xi_p$ discussed earlier~\cite{kjuchukova2018dihedral, cahnkjuchukova2018computing}.  Specifically, these invariants are computed using linking numbers of lifts of curves in $V-\beta$, where $V$ is a Seifert surface for $\alpha$, and $\beta$ is a characteristic knot. For the purposes of this paper, the essential property of a mod~3 characteristic knot $\beta\subset V$ is that every simple closed curve in $V-\beta$ lifts to three closed curves in the dihedral cover of $\alpha$ corresponding to $\beta$. As a result, we have focused on computations with curves in $S^3-\alpha$ whose lifts to a three-fold dihedral cover of $(S^3, \alpha)$ have three connected components.

In the examples below, we let $V$ be the connected sum of two copies of the familiar Seifert surface for the minimal-crossing diagram of the trefoil in 2-bridge position, namely a surface consisting of two disks joined by three twisted bands. The characteristic knot $\beta$ is then the connected sum of two copies of a characteristic knot for the trefoil; it is shown in blue in Figures~\ref{intersectingcurve2.fig} and~\ref{intersectingcurve.fig}.

\subsection{Examples} We apply our algorithm to the following pseudo-branch curves: the characteristic knot $\beta$, defined above; an essential curve $\omega_1$ (see Figure~\ref{intersectingcurve2.fig}) in $V-\beta$, which has one null-homologous lift and two homologically nontrivial lifts; and a pseudo-branch curve $\omega_2$ (see Figure~\ref{intersectingcurve.fig}) which is a push-off of a curve in $V$ intersecting $\beta$ once transversely, and lifts to a single null-homologous closed curve.  
\begin{figure}[htbp]
	\includegraphics[width=6in]{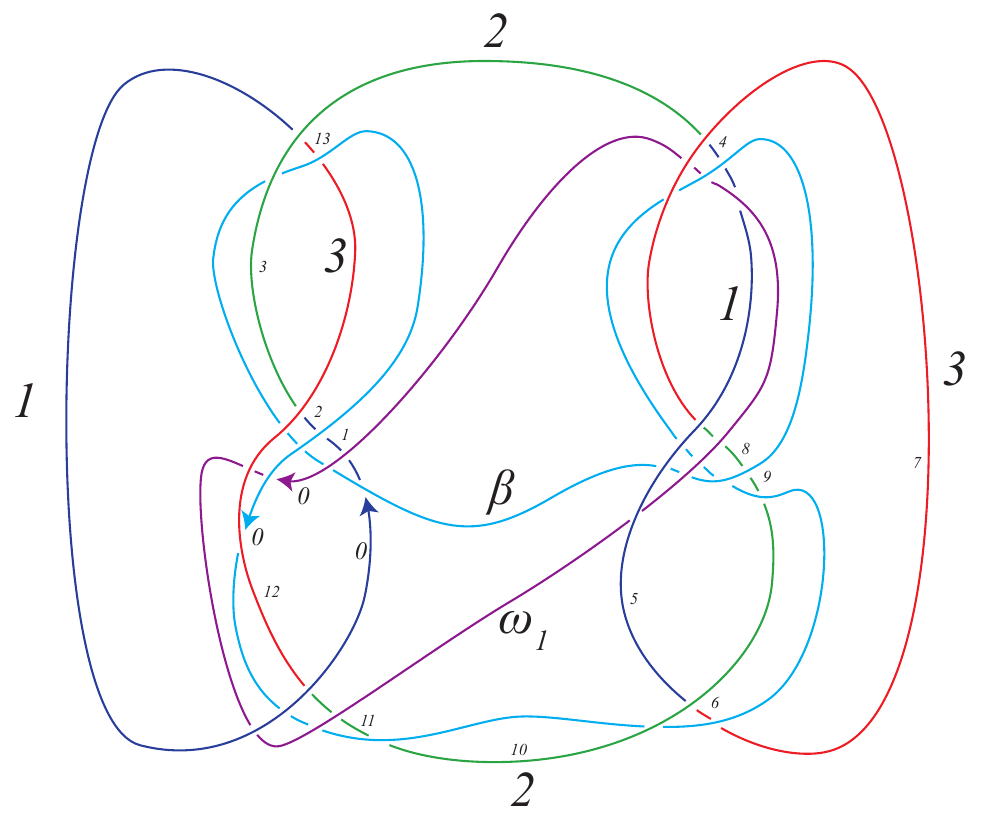}
	\caption{The connected sum, $\alpha$, of two trefoils. A characteristic knot, $\beta$, for $\alpha$. A curve, $\omega_1$, on a Seifert surface $V$ for $\alpha$, which is disjoint from $\beta$.  The numbering on $\alpha$ corresponds to the case where $\beta$ plays the role of the first pseudo-branch curve $\gamma$.}
	\label{intersectingcurve2.fig}
	\end{figure}
Our computer algorithm detects the number of lifts and whether each is rationally null-homologous, and allows us to compute the linking numbers of all pairs of rationally null-homologous lifts. The results of this computation are discussed below.  In each part, we choose one of the curves above to play the role of the first pseudo-branch curve, referred to as $\gamma$ throughout the previous sections (this is the curve for which we find bounding 2-chains), and then compute linking numbers by letting the other curves play the role of the second pseudo-branch curve $\delta$.

{\bf Part I.} To start, the role of the first pseudo-branch curve, denoted by $\gamma$ throughout the previous sections, is played by the characteristic knot $\beta.$

We include all the input needed for the computer program for our first computation, which finds intersection numbers of lifts of $\omega_1$ with 2-chains whose boundaries are lifts of $\beta$.  The input for other computations is similar.

First, we find the list of subscripts corresponding to the overarcs at the end of each arc of $\alpha$:

$$(f(0),f(1),\dots, f(13))=(7,0,12,7,6,10,3,5,6,3,2,0,0,3).$$

Next, we record the color of each arc of $\alpha$,

$$(c(0),c(1),\dots,c(13))=(1,1,1,2,1,1,3,3,2,2,2,2,3,3),$$

and the signs of crossings where arcs of $\alpha$ terminate:

$$(\epsilon(0),\epsilon(1),\dots,\epsilon(13))=(-1,-1,1,1,-1,1,-1,1,1,1,1,1,1,1).$$

We also record whether each arc $k_i$ of $\alpha$ terminates at some other arc of the knot $\alpha$ (in which case we write $t(i)=k$), or at an arc of the first pseudo-branch curve (in which case we write $t(i)=p$); we refer to this as a list of {\it crossing types}:

$$(t(0),t(1),\dots,t(13))=(p,p,k,k,p,k,p,k,p,p,p,k,p,k).$$

Now we record information about the first pseudo-branch curve $\gamma=\beta.$  The subscripts on the overarcs at the end of each arc of $\beta$ are:

$$(f_\gamma(0),f_\gamma(1),\dots,f_\gamma(9))=(12,0,10,6,5,7,5,0,12,3).$$

The signs for $\beta$ are:

$$(\epsilon_\gamma(0),\epsilon_\gamma(1),\dots,\epsilon_\gamma(9))=(1,-1,1,-1,-1,1,-1,-1,1,-1).$$

The list of crossing types for $\beta$ are:

$$(t_\gamma(0),t_\gamma(1),\dots,t_\gamma(9))=(k,k,k,p,k,k,k,p,k,k).$$

The algorithm finds a 2-chain bounding each lift of $\beta$.  The 2-chain bounding the $j^{th}$ lift of $\beta$ can be described by a list of coefficients $x^j_i$ of 2-cells $A_{2,i}$, as defined in Section~\ref{chains.sec}.  The coefficients for the three lifts of $\beta$ are given in Table~\ref{betacoefs.tab}.
\begin{table}[htbp]
\begin{tabular}{|c|c|c|c|c|c|c|c|c|c|c|c|c|c|c|}
	\hline
	$i$&0&1&2&3&4&5&6&7&8&9&10&11&12&13\\
	\hline
	$x^1_i$&-1& -2& -2& -1& -1& 0& 1& 0& 0& 0& 1& 0& -1& 0\\
	\hline
	$x^2_i$ &1& 1 & 2 & 1& 1& 1& 0& 0& -1& 0& -1& 0& 1& 0\\
	\hline
	$x^3_i$ &0& 1 & 0& 0& 0& -1& -1& 0& 1& 0& 0& 0& 0& 0\\ 
	\hline
	
	\end{tabular}
	\caption{The coefficients $x^j_i$ of $A_{2,i}$ in the 2-chain bounding the $j^{th}$ lift of $\beta$.}
	\label{betacoefs.tab}
	\end{table}
	
To compute the intersection numbers, we need to supply the overarc numbers $f_\delta(i)$, signs of crossings $\epsilon_\delta(i)$, and crossing types $t_\delta(i)$ for the second pseudo-branch curve $\delta$.

First, we let $\delta=\omega_1$.  Its over-arc numbers are $(0,12,0,5,6,7)$.  Its signs are $(-1,1,-1,1,1,-1)$.  Its crossing types are $(p,k,k,k,p,k)$.  The matrix of intersection numbers $I_{j,k}$ of a 2-chain bounding the $j^{th}$ lift of $\beta$ with the $k^{th}$ lift of $\omega_1$ is

$$(I_{j,k})=\begin{pmatrix}
{\bf 0}& -1& 1\\
{\bf -1}& 1& 0\\
{\bf 1}& 0& -1
\end{pmatrix}.$$

However, we will see in Part II of this example that only the first lift of $\omega_1$ is null-homologous.  Thus, the first column of the matrix (in bold) gives the linking numbers of the null-homologous lift of $\omega_1$ with each lift of $\beta$. The intersection numbers in the second and third columns turn out not to be well-defined linking numbers.

Next we let $\omega_2$ play the role of the second pseudo-branch curve $\delta$. Accordingly, we input the over-arc numbers $(10,3,6,5)$, signs of crossings $(1,-1,-1,-1)$, and crossing types $(k,p,p,k)$.  The matrix of intersection numbers $I_{j,k}$ of a 2-chain bounding the $j^{th}$ lift of $\beta$ with the $k^{th}$ (path) lift of $\omega_2$ is

$$(I_{j,k})=\begin{pmatrix}-1& -1& 0\\
 -1& 1& -2\\ 
0& -2& 0\end{pmatrix}.$$

In this case the 3 path-lifts of $\omega_2$ fit together to form one closed curve in $S^1\times S^2$.  The linking numbers of the single (closed) lift of $\omega_2$ with each of the 3 lifts of $\beta$ are obtained by summing the rows of the matrix.  Hence, all the linking numbers are $-2$. 

\begin{figure}[htbp]
	\includegraphics[width=6in]{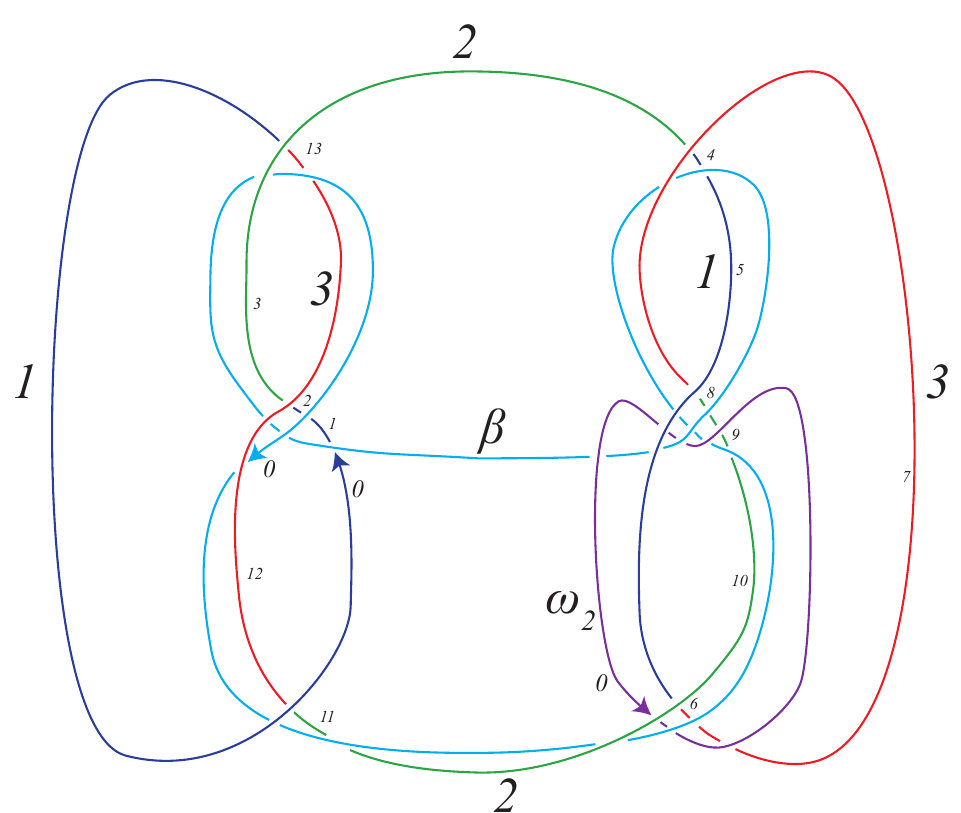}
	\caption{The connected sum, $\alpha$, of two trefoils. A characteristic knot, $\beta$, for $\alpha$. A push-off, $\omega_2$, of a curve on a Seifert surface $V$ for $\alpha$, which intersects $\beta$ once transversely.  The numbering on $\alpha$ corresponds to the case where $\beta$ plays the role of the first pseudo-branch curve $\gamma$.}
	\label{intersectingcurve.fig}
	\end{figure}

{\bf Part II.} To complete the example, we let the role of the first pseudo-branch curve be played by $\omega_1$.

		The list of coefficients $x^j_i$ of the 2-cells $A_{2,i}$ in the 2-chain bounding lift $j$ of $\omega_1$ is given in Table~\ref{betacompcoefs.tab}.  When $j=2,3$ these coefficients are not defined because the corresponding lifts of $\omega_1$ are not null-homologous, and the algorithm detects this, failing to produce a solution for the $x^j_i$.
	
	The matrix of intersection numbers of the 2-chain bounding the $j^{th}$ lift of $\omega_1$ with the $k^{th}$ lift of $\beta$ is
	
	$$(I_{j,k})=\begin{pmatrix}
	0& -1& 1\\
	. &. &. \\
	. & . & .
	\end{pmatrix}.$$
The empty positions in the matrix above indicate that the corresponding rational 2-chain does not exist; i.e., the given lift is not rationally null-homologous.  The first row of the matrix gives the linking numbers of the null-homologous lift of $\omega_1$ with each lift of $\beta$, and we see these numbers agree with the first column of the matrix of intersection numbers of 2-chains bounding lifts of $\beta$ with lifts of $\omega_1$, confirming our first computation.	
	\begin{table}[htbp]

	\begin{tabular}{|c|c|c|c|c|c|c|c|c|c|c|c|c|c|c|}
		\hline
		$i$&0&1&2&3&4&5&6&7&8&9\\
		\hline
		$x^1_i$&0& 0& 0& 0& -1& 0& 1& 1& 0& 0\\
		\hline
		$x^2_i$ &.& . & . & . &.& .& .& .& .& .\\
		\hline
		$x^3_i$ &.& . & . & . &.& .& .& .& .& .\\ 
		\hline

		\end{tabular}
		\caption{The coefficients $x^j_i$ of $A_{2,i}$ in the 2-chain bounding the $j^{th}$ lift of the curve $\omega_1$ in $V-\beta$.  Note that the $x^2_i$ and $x^3_i$ are undefined because the corresponding lifts are not rationally null-homologous. }
		\label{betacompcoefs.tab}
		\end{table}

The algorithm also allows us to compute the linking numbers of each of the null-homologous pseudo-branch curves above (the three lifts of $\beta$; the only null-homologous lift of $\omega_1$; the single closed lift of $\omega_2$) with each of the branch curves as well.  These linking numbers are all zero, as one can also deduce from a geometric argument, using the construction in Cappell-Shaneson \cite{CS1984linking} together with the fact that the curves $\beta$, $\omega_1$, and $\omega_2$ lie on a Seifert surface for $\alpha$.\\

Patricia Cahn\\
Smith College\\
{\it pcahn@smith.edu}

Alexandra Kjuchukova\\
Max Planck Institute for Mathematics -- Bonn\\
{\it sashka@mpim-bonn.mpg.de}
\bibliographystyle{amsplain}
\bibliography{BrCovBib}

\newpage
\section{Appendix: The computer program}

 We used the following input to generate the results above:

\subsection{The characteristic knot $\beta$ is the first pseudo-branch curve.} 

The list of overstrand numbers $f(i)$ for $\alpha$ is $(7,0,12,7,6,10,3,5,6,3,2,0,0,3).$
 
The corresponding list of signs for the knot $\alpha$ is $(-1,-1,1,1,-1,1,-1,1,1,1,1,1,1,1).$
 
The list of crossing types is $(p,p,k,k,p,k,p,k,p,p,p,k,p,k).$
  
The list of colors is $(1,1,1,2,1,1,3,3,2,2,2,2,3,3).$
  
The list of overstrand numbers for the first pseudo-branch curve $\gamma=\beta$ is $(12,0,10,6,5,7,5,0,12,3).$

The corresponding list of signs is $(1,-1,1,-1,-1,1,-1,-1,1,-1).$
 
The list of crossing types is $(k,k,k,p,k,k,k,p,k,k).$

The program returns the matrix 
$$[[-1, -2, -2, -1, -1, 0, 1, 0, 0, 0, 1, 0, -1, 0], [1, 1, 2, 1, 1, 1, 0, 0, -1, 0, -1, 0, 1, 0],$$ $$[0, 1, 0, 0, 0, -1, -1, 0, 1, 0, 0, 0, 0, 0]],$$
 which is the list of coefficients $x^j_i$ of the 2-cells $A_{2,i}$ in the 2-chain bounding lift $i$ of $\beta$.  These coefficients are organized in Table~\ref{betacoefs.tab}.

{\bf $\beta$ is the first pseudo-branch curve and $\omega_1$ is the second pseudo-branch curve.}
The list of overstrand numbers for the second pseudo-branch curve $\omega_1$ is $(0,12,0,5,6,7)$.

The corresponding list of signs is $(-1,1,-1,1,1,-1)$.

The list of crossing types is $(p,k,k,k,p,k)$.

The output of the program is $[[0, -1, 1], [-1, 1, 0], [1, 0, -1]]$.  The interpretation of this matrix is given in Section~\ref{ex}.

{\bf $\beta$ is the first pseudo-branch curve and $\omega_2$ is the second pseudo-branch curve.}

The list of overstrand numbers for the second pseudo-branch cuve $\omega_2$ is $(10,3,6,5)$.

The corresponding list of signs is $(1,-1,-1,-1)$.

The list of crossing types is $(k,p,p,k)$.

The output of the program is $[[-1, -1, 0], [-1, 1, -2], [0, -2, 0]]$. The interpretation of this matrix is given in Section~\ref{ex}.

\subsection{The curve $\omega_1$ in $V-\beta$ is the first pseudo-branch curve.}  

The list of over-crossing numbers $f(i)$ for the subdivided knot diagram is $(0, 9, 5, 3, 7, 4, 3, 2, 0, 2).$

The list of crossing types is $(p, k, k, p, k, k, p, p, k, k).$

The list of colors is $(1, 1, 2, 1, 1, 3, 2, 2, 2, 3).$

The list of signs is $(-1, 1, 1, 1, 1, 1, -1, 1, 1, 1).$

The program returns the matrix 
$$[[0, 0, 0, 0, -1, 0, 1, 1, 0, 0], 'False', 'False'],$$

which is the list of coefficients $x^j_i$ of the 2-cells $A_{2,i}$ in the 2-chain bounding lift $i$ of $\beta$ is given in Table~\ref{betacompcoefs.tab}.  The $'False'$ entries signal that the second and third lifts of $\omega_1$ are not null-homologous.

{\bf $\omega_1$ is the first pseudo-branch curve and $\beta$ is the second pseudo-branch curve.}

The list of overstrand numbers for the second pseudo-branch curve $\beta$ is $(9,0,2,7,3,4,5,3,4,0,9,2)$.

The corresponding list of signs is $(1,-1,-1,1,1,-1,1,1,-1,-1,1,-1)$.

The list of crossing types is $(k,k,p,k,p,k,k,p,k,p,k,k)$.

The output of the program is $[[0, -1, 1], ['x', 'x', 'x'], ['x', 'x', 'x']]$.  The interpretation of this matrix is given in Section~\ref{ex}.

Finally, the Python code is given below.

\subsection{The Python program}  

\begin{verbatim}
# Given an initial 3-cell and the color of a knot arc, returns the number of
# the 3-cell on the other side of the vertical 2-cell below the knot.
def wallcolorchange(oldcolor,wallcolor):
    if oldcolor==wallcolor:
        newcolor=oldcolor
    elif oldcolor != wallcolor:
        s=set()
        s.add(1)
        s.add(2)
        s.add(3)
        s.discard(oldcolor)
        s.discard(wallcolor)
        newcolor=s.pop()
    return newcolor
\end{verbatim}

\begin{verbatim}
# A list of the 3-cells which contain each path-lift of a 
# pseudo-branch curve.
def pseudolifts(subknotcolors,povernums,povertypes):
    l=len(povernums)
    lift1=[1]
    lift2=[2]
    lift3=[3]
    for i in range(0,l-1):
        if povertypes[i]=='k':
            newcell1=wallcolorchange(lift1[i],subknotcolors[povernums[i]])
            lift1.append(newcell1)
            newcell2=wallcolorchange(lift2[i],subknotcolors[povernums[i]])
            lift2.append(newcell2)
            newcell3=wallcolorchange(lift3[i],subknotcolors[povernums[i]])
            lift3.append(newcell3)
        else:
            lift1.append(lift1[i])
            lift2.append(lift2[i])
            lift3.append(lift3[i])
    return lift1, lift2, lift3
\end{verbatim}

\begin{verbatim}
# Returns the vector of values of the function w(i)
def subwhereisA2(subknotcolors,subknottypes,subknotovernums):
    l=len(subknottypes)
    if subknotcolors[0]==1:
        where=[2]
    else:
        where=[1]
    for j in range(0,l-1):
        if subknottypes[j]=='k':
            where.append(wallcolorchange(where[j],
                subknotcolors[subknotovernums[j]]))
        elif subknottypes[j]=='p':
            where.append(where[j])
    return where
\end{verbatim}

\begin{verbatim}
# Returns the value of epsilon_1
def xingsign1(i,subknotcolors,subknottypes,subknotovernums):    
    if subwhereisA2(subknotcolors,subknottypes,subknotovernums)[subknotovernums[i]]
                                                                 !=subknotcolors[i]:
        s=1
    else:
        s=-1
    return s

# Returns the value of epsilon_2
def xingsign2(i,subknotcolors,subknottypes,subknotovernums):
    if subwhereisA2(subknotcolors,subknottypes,subknotovernums)[i]
                                                 !=subknotcolors[subknotovernums[i]]:
        s=-1
    else:
        s=1
    return s

# Returns the value of epsilon_3
def xingsign3(i,subknotcolors,subknottypes,subknotovernums):
    if subwhereisA2(subknotcolors,subknottypes,subknotovernums)[i]
                        ==subwhereisA2(subknotcolors,subknottypes,
					    subknotovernums)[subknotovernums[i]]:
        s=-1
    else:
        s=1
    return s
\end{verbatim}

\begin{verbatim}
# Given j, returns the matrix A|-b, where the solutions to Ax=b are the
# coefficients of the 2-cells A_2i in the chain bounding the j^th lift
# of the first pseudo-branch curve.
def p1surfacecoefmatrix(subknotcolors,subknottypes,subknotovernums,subknotsigns,
                        p1signs,p1overnums,lift):
    n=len(subknotcolors)
    coefmatrix= [[0 for x in range(n+1)] for x in range(n)]   
    for i in range(0,n):
            coefmatrix[i][i]+=1
            coefmatrix[i][(i+1)%n]-=1
            if subknottypes[i]=='k' and
			                      subknotcolors[i]!=subknotcolors[subknotovernums[i]]:
			
			                coefmatrix[i][subknotovernums[i]]+=
							                 xingsign1(i,subknotcolors,subknottypes,subknotovernums)
											                 *xingsign2(i,subknotcolors,subknottypes,subknotovernums)
				
            elif subknottypes[i]=='k' and
			                      subknotcolors[i]==subknotcolors[subknotovernums[i]]:
								  
                coefmatrix[i][subknotovernums[i]]+=
				                   xingsign3(i,subknotcolors,subknottypes,subknotovernums)*2
				
            elif subknottypes[i]=='p' and 
			                         lift[subknotovernums[i]]==subknotcolors[i]:
                coefmatrix[i][n]=0
				
            elif subknottypes[i]=='p' and 
			                         lift[subknotovernums[i]]==
									                  subwhereisA2(subknotcolors,subknottypes,subknotovernums)[i]:
			
                coefmatrix[i][n]=-subknotsigns[i]
				
            elif subknottypes[i]=='p' and lift[subknotovernums[i]]!=
			                subwhereisA2(subknotcolors,subknottypes,subknotovernums)[i]:
							
                coefmatrix[i][n]=subknotsigns[i]
    return coefmatrix
\end{verbatim}

\begin{verbatim}
# Returns the matrix A|-b, where a solution to Ax=b is the vector
# coefficients of the 2-cells A_2i in the chain bounding the index
# one branch curve
def br1surfacecoefmatrix(subknotcolors, subknottypes, subknotovernums, subknotsigns):
    n=len(subknotcolors)
    coefmatrix= [[0 for x in range(n+1)] for x in range(n)]
    
    for i in range(0,n):
        coefmatrix[i][i]+=1
        coefmatrix[i][(i+1)%n]-=1
        if subknottypes[i]=='k' and subknotcolors[i]!=subknotcolors[subknotovernums[i]]:
                coefmatrix[i][subknotovernums[i]]+=xingsign1(i,subknotcolors,
                                                     subknottypes,subknotovernums)*
                                                   xingsign2(i,subknotcolors,
                                                     subknottypes,subknotovernums)
                coefmatrix[i][n]=-subknotsigns[i]*xingsign2(i, subknotcolors,
                                                     subknottypes,subknotovernums)
        elif subknottypes[i]=='k' and subknotcolors[i]==subknotcolors[subknotovernums[i]]:
                coefmatrix[i][subknotovernums[i]]+= xingsign3(i,subknotcolors,
                                                   subknottypes,subknotovernums)*2
    return coefmatrix
\end{verbatim}
\begin{verbatim}
# Returns the matrix A|-b, where a solution to Ax=b is the vector of
# coefficients of the 2-cells A_2i in the chain bounding the index
# two branch curve
def br2surfacecoefmatrix(subknotcolors, subknottypes, subknotovernums, subknotsigns):
    n=len(subknotcolors)
    coefmatrix=[[0 for x in range(n+1)] for x in range(n)]
    
    for i in range(0,n):
        coefmatrix[i][i]+=1
        coefmatrix[i][(i+1)%n]-=1
        if subknottypes[i]=='k' and subknotcolors[i]!=
                                                subknotcolors[subknotovernums[i]]:
            coefmatrix[i][subknotovernums[i]]+=xingsign1(i,subknotcolors,
                                                 subknottypes,subknotovernums)*
                                               xingsign2(i,subknotcolors,
                                                 subknottypes,subknotovernums)
            coefmatrix[i][n]=xingsign2(i, subknotcolors,
                                                 subknottypes,subknotovernums)*.5*
                                               (subknotsigns[i]-xingsign1(i, subknotcolors,
                                                 subknottypes,subknotovernums))

        elif subknottypes[i]=='k' and subknotcolors[i]==
                                                  subknotcolors[subknotovernums[i]]:
            coefmatrix[i][subknotovernums[i]]+= xingsign3(i,subknotcolors,
                                                  subknottypes,subknotovernums)*2
            coefmatrix[i][n]-=xingsign3(i, subknotcolors,
                                                  subknottypes,subknotovernums)
    return coefmatrix
\end{verbatim}
\begin{verbatim}
# Finds a solution to the matrix equation Ax=b given the matrix A|-b
def solvefor2chain(matrixofcoefs, numcrossings):
    M=Matrix(matrixofcoefs)
    pivots=M.rref()[1]
    numpivots=len(pivots)
    RR=M.rref()[0]  
    x=[0 for j in range(numcrossings)]
    
    if numcrossings in pivots:
        return 'False'
    
    else:    
        for i in range(0,numpivots):
            x[pivots[i]]=-RR[i,numcrossings]  
        
        return x
\end{verbatim}

\begin{verbatim}
# Computes the intersection number of the 2nd pseudo-branch curve
# with the 2-chain bounded by the first pseudo-branch curve.
def p2intersectionwithp1surface(subknotcolors, subknotovernums, 
    subknottypes, subknotsigns, p2overnums,p2overtypes, p2signs,
    p1overnums,p1signs,p1overtypes):

    total=[[0,0,0],[0,0,0],[0,0,0]]
    l=len(p2overnums)
    p2lifts=pseudolifts(subknotcolors, p2overnums, p2overtypes)
    p1lifts=pseudolifts(subknotcolors,p1overnums,p1overtypes)
    coeflists=[]
    coeflists.append(solvefor2chain(p1surfacecoefmatrix(subknotcolors,subknottypes,
               subknotovernums,subknotsigns,p1signs,p1overnums,p1lifts[0]),
               len(subknotsigns)))
    coeflists.append(solvefor2chain(p1surfacecoefmatrix(subknotcolors,subknottypes,
               subknotovernums,subknotsigns,p1signs,p1overnums,p1lifts[1]),
               len(subknotsigns)))
    coeflists.append(solvefor2chain(p1surfacecoefmatrix(subknotcolors,subknottypes,
               subknotovernums,subknotsigns,p1signs,p1overnums,p1lifts[2]),
               len(subknotsigns)))
    where=subwhereisA2(subknotcolors,subknottypes,subknotovernums)
    for s in range(0,3): # Three lifts of 1st pseudo-branch curve
        if coeflists[s]!='False':
            for i in range(0,l):# Arcs of the second pseudo-branch curve
                if p2overtypes[i]=='k': 
                    for j in range (0,3):
                        if p2lifts[j][i]==subknotcolors[p2overnums[i]]:
                            total[s][j]+=0
                        elif p2lifts[j][i]==where[p2overnums[i]]:                        
                            total[s][j]+=coeflists[s][p2overnums[i]]
                       
                        elif p2lifts[j][i]!=where[p2overnums[i]]:                    
                            total[s][j]-=coeflists[s][p2overnums[i]]
                       
                if p2overtypes[i]=='p':
                    for j in range (0,3):
                        if p2lifts[j][i]!=p1lifts[s][p2overnums[i]]:
                            total[s][j]+=0
                       
                        elif p2lifts[j][i]==p1lifts[s][p2overnums[i]]:
                            total[s][j]+=p2signs[i]
                        
        else:
            total[s][0]='x'
            total[s][1]='x'
            total[s][2]='x'

return total
\end{verbatim}
\begin{verbatim}
#Compute the intersection number of the first pseudo-branch curve
# with the 2-chain bounded by the index 1 branch curve.
def p1intersectionwithbr1surface(subknotcolors, subknotovernums,
 subknottypes, subknotsigns, p1overnums, p1signs, p1overtypes):
    total=[0,0,0]
    l=len(p1overnums)
    p1lifts=pseudolifts(subknotcolors,p1overnums,p1overtypes)
    coeflist=solvefor2chain(br1surfacecoefmatrix(subknotcolors,
                      subknottypes,subknotovernums,subknotsigns),
                      len(subknotsigns))
    where=subwhereisA2(subknotcolors,subknottypes, subknotovernums)
    if coeflist!='False':
        for i in range(0,l): 
            if p1overtypes[i]=='k': 
                for j in range (0,3):
                    if p1lifts[j]==subknotcolors[p1overnums[i]]:
                        total[j]+=p1signs[i]
                    elif p1lifts[j]==where[p1overnums[i]]:
                        total[j]+=coeflist[p1overnums[i]]
                    elif p1lifts[j]!=where[p1overnums[i]]:
                        total[j]-=coeflist[p1overnums[i]]
            
    else:
        total=['x','x','x']
    return total
\end{verbatim}

\begin{verbatim}
#Compute the intersection number of the first pseudo-branch curve
# with the 2-chain bounded by the index 2 branch curve.
def p1intersectionwithbr2surface(subknotcolors, subknotovernums, subknottypes, 
                                 subknotsigns, p1overnums, p1signs, p1overtypes):
    total=[0,0,0]
    l=len(p1overnums)
    p1lifts=pseudolifts(subknotcolors,p1overnums,p1overtypes)
    coeflist=solvefor2chain(br2surfacecoefmatrix(subknotcolors,subknottypes,
             subknotovernums,subknotsigns),len(subknotsigns))
    where=subwhereisA2(subknotcolors,subknottypes, subknotovernums)
    if coeflist!='False':
        for i in range(0,l): 
            if p1overtypes[i]=='k': 
                for j in range (0,3):
                    if p1lifts[j]==subknotcolors[p1overnums[i]]:
                        total[j]+=0
                    elif p1lifts[j]==where[p1overnums[i]]:
                        total[j]+=coeflist[p1overnums[i]]
                        if p1signs[i]==-1:
                             total[j]+=-1
                    elif p1lifts[j]!=where[p1overnums[i]]:
                        total[j]+=1-coeflist[p1overnums[i]]
                        if p1signs[i]==1:
                             total[j]+=1
    else:
        total=['x','x','x']
    return total
\end{verbatim}

\end{document}